\documentclass[oneside,english,reqno]{amsart}
\usepackage{ae,aecompl}
\usepackage[T1]{fontenc}
\usepackage[latin9]{inputenc}
\usepackage{color}
\usepackage[english]{babel}
\usepackage{textcomp}
\usepackage{mathrsfs}
\usepackage{mathtools}
\usepackage{amstext}
\usepackage{amsthm}
\usepackage{cite}
\usepackage{amssymb}
\usepackage{stmaryrd}
\usepackage{setspace}
\usepackage{enumitem}
\usepackage{adjustbox}

\newcommand\mathtight{
	\thinmuskip=1mu
	\medmuskip=2mu
	\thickmuskip=4mu
	\renewcommand\quad{\hspace*{0.5em}}
	\renewcommand\qquad{\quad\quad}
}%

\usepackage[unicode=true,pdfusetitle, bookmarks=true,bookmarksnumbered=false,bookmarksopen=false, breaklinks=false,pdfborder={0 0 0},backref=false,colorlinks=true,linkcolor=blue] {hyperref}
 \usepackage{algorithm}
 \usepackage[noend]{algpseudocode}
 \usepackage[all]{hypcap} 
\usepackage{cleveref}

\usepackage{ifpdf}
\usepackage{color}

\makeatletter
\newcounter{algorithmicH}
\let\oldalgorithmic\algorithmic
\renewcommand{\algorithmic}{%
  \stepcounter{algorithmicH}
  \oldalgorithmic}
\renewcommand{\theHALG@line}{ALG@line.\thealgorithmicH.\arabic{ALG@line}}
\makeatother

\newcommand{\dom}{\mathop{\rm dom}\nolimits}

\newcommand{\id}{\mathop{\rm Id}\nolimits}

\newcommand{\prox}{\mathop{\rm prox}\nolimits}
\newcommand{\fix}{\mathop{\rm Fix}\nolimits}

\newcommand{\minimize}{\operatorname{minimize}}  
\newcommand{\maximize}{\operatorname{maximize}}  

\newcommand{\Nn}{{\rm{I\!N}}}

\newcommand{\HH}{\mathcal H}
\newcommand{\KK}{\mathcal K}

\newcommand{\ie}{\emph{i.e.}}

\newcommand{\zer}{\mathop{\rm zer}\nolimits}
\newcommand{\gra}{\mathop{\rm gra}\nolimits}
\newcommand{\ran}{\mathop{\rm ran}\nolimits}

\DeclareMathOperator*{\argmin}{\arg\!\min}
\newcommand{\stt}{\rm subject\ to}

\newtheorem{thm}{Theorem}[section]
\newtheorem{lem}{Lemma}[section]
\newtheorem{prop}{Proposition}[section]

\theoremstyle{remark}
\newtheorem{rem}{Remark}[section]

\theoremstyle{definition}
\newtheorem{ass}{Assumption}[section]

\newtheorem{deff}{Definition}[section]

\crefname{thm}{Thm.}{Thm.}
\Crefname{thm}{Theorem}{Theorems}
\crefformat{thm}{#2Thm. #1#3}
\Crefformat{thm}{#2Theorem #1#3}
\crefmultiformat{thm}{#2Thm. #1#3}{ and #2#1#3}{, #2#1#3}{ and #2#1#3}
\Crefmultiformat{thm}{#2Theorems #1#3}{ and #2#1#3}{, #2#1#3}{ and #2#1#3}

\crefname{prop}{Prop.}{Prop.}
\Crefname{prop}{Proposition}{Propositions}
\crefformat{prop}{#2Prop. #1#3}
\Crefformat{prop}{#2Proposition #1#3}
\crefmultiformat{prop}{#2Prop. #1#3}{ and #2#1#3}{, #2#1#3}{ and #2#1#3}
\Crefmultiformat{prop}{#2Propositions #1#3}{ and #2#1#3}{, #2#1#3}{ and #2#1#3}

\crefname{coro}{Cor.}{Cro.}
\Crefname{coro}{Corollary}{Corollaries}
\crefformat{coro}{#2Cor. #1#3}
\Crefformat{coro}{#2Corollary #1#3}
\crefmultiformat{coro}{#2Cor. #1#3}{ and #2#1#3}{, #2#1#3}{ and #2#1#3}
\Crefmultiformat{coro}{#2Corollaries #1#3}{ and #2#1#3}{, #2#1#3}{ and #2#1#3}

\crefname{lem}{Lem.}{Lem.}
\Crefname{lem}{Lemma}{Lemmas}
\crefformat{lem}{#2Lem. #1#3}
\Crefformat{lem}{#2Lemma #1#3}
\crefmultiformat{lem}{#2Lem. #1#3}{ and #2#1#3}{, #2#1#3}{ and #2#1#3}
\Crefmultiformat{lem}{#2Lemmas #1#3}{ and #2#1#3}{, #2#1#3}{ and #2#1#3}

\crefname{algorithm}{Alg.}{Alg.}
\Crefname{algorithm}{Algorithm}{Algorithms}
\crefformat{algorithm}{#2Alg. #1#3}
\Crefformat{algorithm}{#2Algorithm #1#3}
\crefmultiformat{algorithm}{#2Alg. #1#3}{ and #2#1#3}{, #2#1#3}{ and #2#1#3}
\Crefmultiformat{algorithm}{#2Algorithms #1#3}{ and #2#1#3}{, #2#1#3}{ and #2#1#3}

\crefname{rem}{Rem.}{Rem.}
\Crefname{rem}{Remark}{Remarks}
\crefformat{rem}{#2Rem. #1#3}
\Crefformat{rem}{#2Remark #1#3}
\crefmultiformat{rem}{#2Rem. #1#3}{ and #2#1#3}{, #2#1#3}{ and #2#1#3}
\Crefmultiformat{rem}{#2Remarks #1#3}{ and #2#1#3}{, #2#1#3}{ and #2#1#3}

\crefname{ass}{Ass.}{Ass.}
\Crefname{ass}{Assumption}{Assumption}
\crefformat{ass}{#2Ass. #1#3}
\Crefformat{ass}{#2Assumption #1#3}

\crefname{deff}{Def.}{Def.}
\Crefname{deff}{Definition}{Definition}
\crefformat{deff}{#2Def. #1#3}
\Crefformat{deff}{#2Definition #1#3}

\Crefname{section}{Section}{Sections}
\Crefformat{section}{#2Section #1#3}
\Crefmultiformat{section}{#2Sections #1#3}{ and #2#1#3}{, #2#1#3}{ and #2#1#3}

\Crefname{enumi}{Assumption}{Assumptions}
\Crefformat{enumi}{Assumption #21(#1)#3}
\Crefmultiformat{enumi}{Assumptions (#2#1#3)}{ and (#2#1#3)}{, (#2#1#3)}{ and (#2#1#3)}
\crefrangeformat{enumi}{Assumptions ~(#3#1#4) to~(#5#2#6)}

\Crefname{figure}{Figure}{Figures}
\Crefformat{figure}{#2Figure #1#3}
 
\newcommand\proofRef[2][Appendix]{\hyperref[#2]{#1}}
\newenvironment{Proof}[1]{\begin{proof}[#1]}{\end{proof}}
 \usepackage[foot]{amsaddr}

\begin{document}
\title[Asymmetric Forward-Backward-Adjoint Splitting]{Asymmetric Forward-Backward-Adjoint Splitting for solving monotone inclusions involving three operators}
\author[P. Latafat]{Puya Latafat}
\address[P. Latafat]{IMT School for Advanced Studies Lucca, Piazza San Francesco 19, 55100 Lucca, Italy}
\email{puya.latafat@imtlucca.it}
\author[P. Patrinos]{Panagiotis Patrinos}
\address[P. Patrinos]{Department of Electrical Engineering (ESAT-STADIUS), KU Leuven, Kasteelpark  Arenberg 10, 3001 Leuven-Heverlee, Belgium}
\email{panos.patrinos@esat.kuleuven.be}

\begin{abstract}

In this work we propose a new splitting technique, namely Asymmetric Forward-Backward-Adjoint splitting,
for solving monotone inclusions involving three terms, a maximally monotone, a cocoercive and a bounded linear operator. Classical operator splitting methods, like  Douglas-Rachford and Forward-Backward splitting are special cases of our new algorithm. Asymmetric Forward-Backward-Adjoint splitting  unifies, extends and sheds light on the connections between many seemingly unrelated primal-dual algorithms for solving structured convex optimization problems proposed in recent years. More importantly, it greatly extends the scope and applicability of splitting techniques to a wider variety of problems. One important special case leads to a Douglas-Rachford type scheme that includes a third cocoercive operator. 
\end{abstract}
\maketitle
\section{Introduction} \label{sec:1}

 This paper considers two types of general problems. The focus of the first part of the paper is on 
 solving monotone inclusion problems of the form 
\begin{align}
0\in Ax+Mx+Cx, \label{eq:main inclusion}
\end{align}
where $A$ is a maximally monotone operator, $M$ is a bounded linear operator and $C$ is cocoercive\footnote{ $C$
is $\beta$-cocoercive with respect to the $P$ norm if for some $\beta\in ]0,+\infty[$ the following holds 
		\[
		(\forall z\in\HH)(\forall z^{\prime}\in\HH)\quad\langle Cz-Cz^{\prime},z-z^{\prime}\rangle\geq \beta\|Cz-Cz^{\prime}\|_{P^{-1}}^{2}.
		\]}. The most well known algorithms for solving monotone inclusion problems are Forward-Backward splitting (FBS), Douglas-Rachford splitting (DRS) and Forward-Backward-Forward splitting (FBFS) \cite{moreau1965proximite,lions1979splitting,combettes2011proximal,parikh2013proximal,boyd2011distributed,tseng2000modified}. The operator splitting schemes FBS and DRS are not well equipped to handle~\eqref{eq:main inclusion} since they are designed for monotone inclusions involving the sum of two operators. The FBFS splitting can solve~\eqref{eq:main inclusion} by considering $M+C$ as one Lipschitz continuous operator. However, being blind to the fact that $C$ is cocoercive, it would require two evaluations of $C$ per iteration. Many other variations of the three main splittings have been proposed over time that can be seen as  intelligent applications of these classical methods (see for example \cite{briceno2015forward,chambolle2011first,briceno2011monotone+,vu2013splitting,condat2013primal}).

The main contribution of the paper is a new algorithm called \emph{Asymmetric-Forward-Backward-Adjoint splitting} (AFBA) to solve the monotone inclusion~\eqref{eq:main inclusion}, without resorting to any kind of reformulation of the problem. One important property of AFBA is that it includes asymmetric preconditioning. This gives great flexibility to the algorithm, and indeed it is the key for recovering and unifying existing primal-dual proximal splitting schemes for convex optimization and devising new ones. More importantly, it can deal with problems involving 3 operators, one of which is cocoercive. It is observed that  FBS, DRS, the Proximal Point Algorithm (PPA) can be derived as special cases of our method. Another notable special case is the method proposed by Solodov and Tseng for variational inequalities in~\cite[Algorithm 2.1]{solodov1996modified}. Moreover,  when the cocoercive term, $C$, is absent in~\eqref{eq:main inclusion}, in a further special case, it coincides with the FBFS when its Lipschitz operator is skew-adjoint. Recently a new splitting scheme was proposed in~\cite{davis2015three} for solving monotone inclusions involving the sum of three operators, one of which is cocoercive. This method can be seen as Douglas-Rachford splitting with an extra forward step for the cocoercive operator and at this point it seems that it can not be derived by manipulating one of the main three splitting algorithms. As a special case of our scheme, we propose an algorithm that also bares heavy resemblance to the classic Douglas-Rachford splitting with an extra forward step (see \Cref{Algorithm-6}). The proposed algorithm is different than the one of~\cite{davis2015three}, in that the forward step precedes the two backward updates.

As another contribution of the paper, big-$O(1/(n+1))$ and little-$o(1/(n+1))$ convergence rates are derived for AFBA (see \Cref{Thm: conv-rates}). It is observed that in many cases these convergence rates are guaranteed under mild conditions. In addition, under metric subregularity of the underlying operator, linear convergence is guaranteed without restrictions on the parameters (see \Cref{Thm: conv-rates-ii}). Given that AFBA generalizes a wide range of algorithms, this analysis provides a systematic way to deduce convergence rates for many algorithms. 

The focus of the second half of the paper, in the simpler form, is on solving convex optimization problems of the form   
\begin{align}
\underset{x\in\HH}{\minimize}\ f(x)+g(Lx)+h(x), \label{eq:convex}
\end{align}
where $\HH$ is a real Hilbert space and $L$ is a bounded linear operator. The functions $f$, $g$ and $h$ are proper, closed convex functions and in addition $h$ is Lipschitz differentiable. 
The equivalent monotone inclusion problem takes the form of finding $x\in\HH$ such that
\begin{align*}
0\in Ax+L^{*} B Lx+Cx, 
\end{align*}
where $A=\partial f$, $B=\partial g$ are maximally monotone operators defined on $\HH$ and the operator $C=\nabla h $ is cocoercive.
 Similar to the problem~\eqref{eq:main inclusion}, the classical methods FBS, DRS, FBFS are not suitable for solving problems of the form~\eqref{eq:convex} (without any reformulation) because they all deal with problems involving two operators. Furthermore, these methods usually require calculation of the proximal mapping of the composition of a function with a linear operator  which is not trivial in general or requires matrix inversion (see~\cite{parikh2013proximal} for a survey on proximal algorithms). In the recent years in order to solve problem~\eqref{eq:convex}, with or without the cocoercive term, many authors have considered the corresponding saddle point problem. This approach yields the primal and dual solutions simultaneously (hence the name primal-dual splittings) and eliminates the need to calculate the proximal mapping of a linearly composed function. The resulting algorithms only require matrix vector products, gradient and proximal updates (see \cite{combettes2012primal,chambolle2011first,condat2013primal,vu2013splitting,boct2014recent} for more discussion). We follow the same approach and notice that it is quite natural to embed the optimality condition of the saddle point problem associated to \eqref{eq:convex} in the form of the monotone inclusion \eqref{eq:main inclusion}. Subsequently, by appealing to AFBA, we can generate new algorithms and recover many existing methods, such as the ones proposed in \cite{chambolle2011first,condat2013primal,vu2013splitting,briceno2011monotone+,drori2015simple,he2012convergence}, as special cases. In many of the cases, we extend the range of acceptable stepsizes and relaxation parameters under which the methods are convergent. Additionally, the convergence rates for these methods are implied by our results for AFBA. 

The paper is organized as follows. \Cref{sec:2} is devoted to introducing notation and reviewing basic definitions. In~\Cref{sec:3}, we present and analyze the convergence and rate of convergence of AFBA. Its relation to classical splitting methods is discussed in \Cref{sec:special cases}.  In~\Cref{sec:Saddle-point-problem}, we consider the saddle point problem associated to a generalization of~\eqref{eq:convex}. By applying AFBA and properly choosing the parameters we are able to generate a large class of algorithms. We then consider some important special cases and discuss their relation to existing methods. These connections are summarized in the form of a diagram in~\Cref{fig1}.

\section{Backround and Preliminary Results} \label{sec:2}
In this section we recap the basic definitions and results that will be needed subsequently (see \cite{bauschke2011convex} for detailed discussion).\\
Let $\mathcal{H}$ and $\mathcal{G}$ be real Hilbert spaces. We denote the scalar product and the induced norm of a Hilbert space by $\langle\cdot,\cdot\rangle$ and $\|\cdot\|$ respectively. $\id$  denotes the identity operator. We denote
by $\mathscr{B}(\HH,\mathcal{G})$ the space of bounded linear operators
from $\HH$ to $\mathcal{G}$ and set $\mathscr{B}(\HH)=\mathscr{B}(\HH,\mathcal{\HH})$.
The space of self-adjoint operators is denoted by \linebreak $\mathcal{S}(\HH)=\{L\in\mathscr{B}(\HH)|L=L^{*}\}$,
where $L^{*}$ denotes the adjoint of $L$. The \emph{Loewner partial
	ordering} on $\mathcal{S}(\HH)$ is denoted by $\succeq$. 
Let $\tau\in]0,+\infty[$ and define the space of $\tau$-strongly positive
self-adjoint operators by $\mathcal{S}_\tau(\HH)=\{U\in\mathcal{S}(\HH)|U\succeq\tau\id\}.$
For $U\in\mathcal{S}_\tau(\HH)$, define the scalar product and norm by $\langle x,y\rangle_{U}=\langle x,Uy\rangle$,
and \linebreak ${\mathtight\|x\|_{U}=\sqrt{\langle x,Ux\rangle}}$. We also define the Hilbert space $\HH_{U}$ by endowing $\HH$ with the scalar product $\langle x,y\rangle_{U}$. 
One has ${\mathtight |\langle x,y\rangle|\leq\|x\|_{U}\|y\|_{U^{-1}}}$
and ${\mathtight\|x\|_{U}=\sqrt{\langle x,Ux\rangle}}$. The operator norm induced by $\|\cdot\|_{U}$ is ${\mathtight\|L\|_{U}=\sup
_{x\in\HH,\|x\|_{U}=1}\|Lx\|_{U}=\|U^{1/2}LU^{-1/2}\|}$. 


Let $A:\HH\to2^{\HH}$ be a set-valued operator. The domain of $A$ is denoted by $\dom (A)=\{x\in\HH|Ax\neq\textrm{Ø}\}$, its graph by $\gra (A)=\{(x,u)\in\HH\times\HH|u\in Ax\}$ and the set of zeros of $A$ is $\zer (A)=\{x\in\HH|0\in Ax\}$. The inverse of $A$ is defined through its graph: $ 
\gra (A^{-1})=\{(u,x)\in\HH\times\HH|(x,u)\in \gra (A)\}
$.
The \emph{resolvent} of $A$ is given by $J_{A}=(\id +A)^{-1}$.  Furthermore, $A$ is \emph{monotone} if $\langle x-y,u-v\rangle \geq 0$ for all 
$(x,u),(y,v)\in \gra (A)$,
and \emph{maximally monotone} if it is monotone and there exists no monotone operator $B:\HH\to2^{\HH}$ such that $\gra (A) \subset \gra (B)$ and $A\neq B$. 

The set of proper lower semicontinuous convex functions from $\HH$
to ${\mathtight ]-\infty,+\infty]}$ 
is denoted by $\Gamma_{0}(\HH)$. The \emph{Fenchel conjugate}
of  $f\in\Gamma_{0}(\HH)$, denoted $f^{*}\in\Gamma_{0}(\HH)$, is defined by $f^{*}:\HH\to]-\infty,+\infty]:u\mapsto{\sup}_{z\in\HH}\left(\langle u,z\rangle-f(z)\right)$.
The \emph{Fenchel-Young inequality}, $\langle x,u \rangle \leq f(x)+f^*(u)$ for all $x,u\in\HH$, 
holds for $f:\HH\to]-\infty,+\infty]$ proper. Throughout this paper we make extensive use of this inequality in the special case when $f=\frac{1}{2}\|\cdot\|^2_U$ for $U$ strongly positive. The \emph{infimal convolution} of $f,g:\HH\rightarrow]-\infty,+\infty]$  is denoted by 
${\mathtight f\oblong g:\HH\rightarrow]\infty,+\infty]:x\mapsto{\inf_{y\in \HH}}\left(f(y)+g(x-y)\right).
}$
If $f\in\Gamma_{0}(\HH)$ then the \emph{subdifferential} of $f$, denoted
by $\partial f$, is the maximally monotone operator $\partial f:\HH\rightarrow2^{\HH}:x\mapsto\{u\in\HH|(\forall y\in\HH)\,\langle y-x,u\rangle+f(x)\leq f(y)\}$, 
with inverse $(\partial f)^{-1}=\partial f^{*}$.
The resolvent  of $\partial f$ is called \emph{proximal operator} and is uniquely determined by 
$\prox_f(x) \coloneqq J_{\partial f}(x)= \argmin_{z\in \HH} f(z)+\frac{1}{2} \|x-z\|^2$.

Let $X$ be a nonempty closed convex set in $\HH$. The \emph{indicator function}  of $X$, denoted ${\mathtight \iota_X:\HH\to]-\infty,+\infty]}$, is defined by $\iota_X(x)=0$, if $x\in X$ and $\iota_X(x)=+\infty$, otherwise. 
The \emph{normal cone} of $X$ is the maximally monotone operator ${\mathtight N_X\coloneqq\partial \iota_X}$. 
The distance to $X$ with respect to $\|\cdot\|_U$ is denoted by $d_U(z,X)=\inf_{z^{\star}\in X}\|z-z^{\star}\|_U$, the projection of $z$ onto $X$ with respect to $\|\cdot\|_U$ is denoted by $\Pi^U_X(z)$, and the absence of superscript implies the same definitions with respect to the canonical norm. 

\section{{Asymmetric Forward-Backward-Adjoint Method}} \label{sec:3}
Let $\mathcal{H}$ be a real Hilbert space and consider the problem of finding $z\in \HH$ such that 
\begin{equation}
0\in Tz \quad \textrm{where} \quad T\coloneqq A+M+C, \label{eq:}
\end{equation}
where operators $A$, $M$, $C$ satisfy the following assumption:
\begin{ass} \label{assumption-1}
	Throughout the paper the following hold:
	\begin{enumerate}[{label=(\textit{\roman*})},ref=\textit{\roman{*}}]
		\item \label{enu:-is-maximally} Operator $A:\HH\to2^{\HH}$ is maximally monotone and $M\in\mathscr{B}(\HH)$ is monotone.
		\item \label{enu:Let--and} Operator $C:\HH\to\HH$ is $\beta$-cocoercive with respect to $\|\cdot\|_P$, where
		$\beta\in]1/4,+\infty[$ and $P\in\mathcal{S}_\rho(\HH)$ for some $\rho\in]0,\infty[$, \ie
		\[
		(\forall z\in\HH)(\forall z^{\prime}\in\HH)\quad\langle Cz-Cz^{\prime},z-z^{\prime}\rangle\geq \beta\|Cz-Cz^{\prime}\|_{P^{-1}}^{2}.
		\]
	\end{enumerate}
\end{ass}
It is important to notice that the freedom in choosing $P$ is a crucial part of our method. In \Cref{assumption-1}\eqref{enu:Let--and} we consider cocoercivity with respect to ${\|\cdot\|_P}$ with $\beta\in]1/4,+\infty[$. However, this is by no means a restriction of our setting; another approach would have been to consider cocoercivity with respect to the canonical norm $\|\cdot\|$ with $\beta\in]0,+\infty[$ but this would lead to statements involving $\|P\|$ and $\|P^{-1}\|$. 
Indeed convergence with respect to $\|\cdot\|$ and $\|\cdot\|_{P}$ are equivalent but in using $\|\cdot\|_{P}$ we simplify the notation substantially. 

In addition, let $S$ be a strongly positive, self-adjoint operator, $K\in\mathscr{B}(\HH)$ a
skew-adjoint operator, \emph{i.e.}, $K^{*}=-K$ and $H=P+K$. Then, the algorithm for solving the monotone inclusion described above is as follows:
\begin{algorithm}[H]
    \caption{Asymmetric Forward-Backward-Adjoint Splitting (AFBA)}
    \label{Algorithm-1}
    \begin{algorithmic} 
            \item\textbf{Inputs:} $z_0\in\HH$
            \For{$n=0,\ldots,$} 
                \State \hypertarget{alg:line1}{$\bar{z}_{n} =(H+A)^{-1}(H-M-C)z_{n}$}
                \State $\tilde{z}_{n} =\bar{z}_{n}-z_{n}$
                \State \hypertarget{alg:line3}{
                $
                	\displaystyle
                	\alpha_n
                	=
                	\frac%
                	{
                		\lambda_n
                		\|\tilde z_n\|_P^2
                	}
                	{
                		\|
                			\left(
                				H+M^*
                			\right)
                			\tilde z_n
                		\|_{S^{-1}}^2
                	}}
                $
                \State \hypertarget{alg:line4}{ $z_{n+1}=z_{n}+\alpha_{n}S^{-1}(H+M^*)\tilde{z}_{n}$ \label{eq:line4}}
            \EndFor
    \end{algorithmic}
\end{algorithm}
Before proceeding with the convergence analysis let us define 
\begin{equation}\label{eq:D}
D=(H+M^*)^*S^{-1}(H+M^*).
\end{equation}
Since $P\in\mathcal{S}_\rho(\HH)$ for some $\rho\in]0,\infty[$, $K$ is skew-adjoint, and $M\in\mathscr{B}(\HH)$ is monotone, it follows that $\langle (H+M^*)z,z\rangle\geq \rho\|z\|^2$ for all $z\in\HH$
, and we have 
\begin{equation} \label{eq:D pos}
(\forall z\in\HH)\quad\langle z,Dz\rangle=\|(H+M^*)z\|_{S^{-1}}^2 \geq \rho^2\|S\|^{-1} \|z\|^2.
\end{equation}
Hence, $D\in\mathcal{S}_\nu(\HH)$ with $\nu=\rho^2\|S\|^{-1}$. Notice that the denominator of $\alpha_n$ in \Cref{Algorithm-1} is equal to the left hand side of~\eqref{eq:D pos} for $z=\tilde{z}_n$ and thus it is bounded below by $\rho^2\|S\|^{-1} \|\tilde{z}_n\|^2$.

\subsection{{Convergence Analysis}} \label{sec:4}
In this section we analyze convergence and rate of convergence of \Cref{Algorithm-1}. We also consider a special case of the algorithm in which it is possible to relax strong positivity of $P$ to positivity. We begin by stating our main convergence result. 
The proof relies on showing that the sequence $(z_{n})_{n\in\Nn}$ is F\'{e}jer monotone with respect to $\zer (A+M+C)$ in the Hilbert space $\HH_{S}$.

\begin{thm} \label{thm:Suppose-that-}
\begin{singlespace}
 Consider \Cref{Algorithm-1} under \Cref{assumption-1} and assume $\zer (T)\neq\textrm{\O}$ where $T=A+M+C$. Let $\sigma\in]0,\infty[$, $S\in\mathcal{S}_\sigma(\HH)$,  $K\in\mathscr{B}(\HH)$ a
 skew-adjoint operator, and $H=P+K$. Let $(\lambda_{n})_{n\in\Nn}$ be a sequence such that
\begin{equation}
(\lambda_{n})_{n\in\Nn}\subseteq[0,\delta] \quad \textrm{with} \quad \delta=2-\frac{1}{2\beta},\quad\delta>0,\quad \underset{n\rightarrow\infty}{\lim\inf} \,\lambda_{n}(\delta-\lambda_{n})>0. \label{eq:-7}
\end{equation} 
Then the following hold:
\begin{enumerate}[{label=(\textit{\roman*})},ref=\textit{\roman{*}}]

\item \label{thm-part1} $(z_{n})_{n\in\Nn}$ is F\'{e}jer monotone with respect to $\zer (T)$ in the Hilbert space $\HH_{S}$.
\item \label{thm-part2}$(\tilde{z}_{n})_{n\in\Nn}$ converges strongly to zero.
\item \label{thm-part3} $(z_{n})_{n\in\Nn}$ converges weakly to a point in $\zer (T)$.
\end{enumerate}
Furthermore, when $C\equiv0$ all of the above statements hold with $\delta=2$.
	\end{singlespace}
\end{thm}
\begin{proof}
The operators $\tilde{A}=P^{-1}(A+K)$ and $\tilde{B}=P^{-1}(M+C-K)$
are monotone in the Hilbert space $\HH_{P}$. We observe that
\[
\bar{z}_{n}=\left(H+A\right)^{-1}(H-M-C)z_{n}=(\id+\tilde{A})^{-1}(\id-\tilde{B})z_{n}.
\]
Therefore $z_{n}-\tilde{B}z_{n}\in\bar{z}_{n}+\tilde{A}\bar{z}_{n}$,
or $-\tilde{z}_{n}-\tilde{B}z_{n}\in\tilde{A}\bar{z}_{n}$. Since $-\tilde{B}z^{\star}\in\tilde{A}z^{\star}$
for $z^{\star}\in\zer(T)$ by monotonicity of $\tilde{A}$ on $\HH_{P}$ we have 
\begin{align}
\langle\tilde{B}z_{n}-\tilde{B}z^{\star}+\tilde{z}_{n},z^{\star}-\bar{z}_{n}\rangle_{P} \geq0. \nonumber
\end{align}
Then, 
\begin{align}
0 & \leq\langle\tilde{B}z_{n}-\tilde{B}z^{\star}+\tilde{z}_{n},z^{\star}-\bar{z}_{n}\rangle_{P}\nonumber \\
 & =\langle P^{-1}(M+C-K)z_n-P^{-1}(M+C-K)z^\star	+\tilde{z}_{n},z^{\star}-\bar{z}_{n}\rangle_P.\nonumber \\
 &  =\langle(M-K)(z_{n}-z^{\star})+Cz_{n}-Cz^{\star}+P\tilde{z}_{n},z^{\star}-\bar{z}_{n}\rangle.\label{eq:-3-2}
\end{align}
On the other hand
\begin{align*}
\langle Cz_{n}-Cz^{\star},z^{\star}-\bar{z}_{n}\rangle & =\langle Cz_{n}-Cz^{\star},z_{n}-\bar{z}_{n}\rangle+\langle Cz_{n}-Cz^{\star},z^{\star}-z_{n}\rangle\\
 & \leq\frac{\epsilon}{2}\|\tilde{z}_{n}\|_{P}^{2}+\frac{1}{2\epsilon}\|Cz_{n}-Cz^{\star}\|_{P^{-1}}^{2}+\langle Cz_{n}-Cz^{\star},z^{\star}-z_{n}\rangle\\
 & \leq\frac{\epsilon}{2}\|\tilde{z}_{n}\|_{P}^{2}+\left(1-\frac{1}{2\epsilon\beta}\right)\langle Cz_{n}-Cz^{\star},z^{\star}-z_{n}\rangle.
\end{align*}
The first inequality follows from  Fenchel-Young inequality for $\frac{\epsilon}{2}\|\cdot\|_P^2$, while the second   from $\beta$-cocoercivity
of $C$ with respect to $\|\cdot\|_P$. Set $\epsilon\coloneqq\frac{1}{2\beta}$ so that
\begin{align}
\langle Cz_{n}-Cz^{\star},z^{\star}-\bar{z}_{n}\rangle & \leq\frac{1}{4\beta}\|\tilde{z}_{n}\|_{P}^{2}.\label{eq:-4-2}
\end{align}
In turn, \eqref{eq:-3-2}, \eqref{eq:-4-2} and monotonicity of $M-K$,
yield
\begin{align*}
0 & \leq\langle(M-K)(z_{n}-z^{\star})+Cz_{n}-Cz^{\star}+P\tilde{z}_{n},z^{\star}-\bar{z}_{n}\rangle\\
 & \leq \langle(M-K)(z_{n}-z^{\star}),z^{\star}-z_{n}\rangle+\langle(M-K)(z_{n}-z^{\star}),z_{n}-\bar{z}_{n}\rangle\\
 & +\frac{1}{4\beta}\|\tilde{z}_{n}\|_{P}^{2}+\langle P\tilde{z}_{n},z^{\star}-z_{n}\rangle+\langle P\tilde{z}_{n},z_{n}-\bar{z}_{n}\rangle\\
 & \leq\langle z_{n}-z^{\star},(M^{*}+K)(z_{n}-\bar{z}_{n})\rangle+\frac{1}{4\beta}\|\tilde{z}_{n}\|_{P}^{2}+\langle P\tilde{z}_{n},z^{\star}-z_{n}\rangle-\|\tilde{z}_{n}\|_{P}^{2}\\
 & =\langle z_{n}-z^{\star},-(M^{*}+H)\tilde{z}_{n}\rangle-\left(1-\frac{1}{4\beta}\right)\|\tilde{z}_{n}\|_{P}^{2},
\end{align*}
or equivalently 
\begin{equation}
\langle z_{n}-z^{\star},(M^{*}+H)\tilde{z}_{n}\rangle\leq-\left(1-\frac{1}{4\beta}\right)\|\tilde{z}_{n}\|_{P}^{2}. \label{eq:42-42-11}
\end{equation}
For notational convenience define $\delta\coloneqq2-\frac{1}{2\beta}$. We show that $\|z_{n}-z^{\star}\|^2_S$ is decreasing using \eqref{eq:42-42-11} together with \hyperlink{alg:line3}{step 3} and \hyperlink{alg:line3}{4} of \Cref{Algorithm-1}:
\begin{align}
\|z_{n+1}-z^{\star}\|_{S}^{2} & =\|z_{n}-z^{\star}+\alpha_{n}S^{-1}(H+M^{*})\tilde{z}_{n}\|_{S}^{2} \nonumber \\
 & =\|z_{n}-z^{\star}\|_{S}^{2}+2\alpha_{n}\langle z_{n}-z^{\star},(H+M^{*})\tilde{z}_{n}\rangle+\alpha_{n}^{2}\|(H+M^{*})\tilde{z}_{n}\|_{S^{-1}}^{2} \nonumber \\
 & \leq\|z_{n}-z^{\star}\|_{S}^{2}-\alpha_{n}\delta\|\tilde{z}_{n}\|_{P}^{2}+\alpha_{n}^{2}\|(H+M^{*})\tilde{z}_{n}\|_{S^{-1}}^{2}\nonumber \\
 & =\|z_{n}-z^{\star}\|_{S}^{2}-\delta\lambda_{n}\frac{\|\tilde{z}_{n}\|_{P}^{4}}{\|(H+M^{*})\tilde{z}_{n}\|_{S^{-1}}^{2}}+\lambda_{n}^{2}\frac{\|\tilde{z}_{n}\|_{P}^{4}}{\|(H+M^{*})\tilde{z}_{n}\|_{S^{-1}}^{2}}\nonumber \\
 & =\|z_{n}-z^{\star}\|_{S}^{2}-\lambda_{n}(\delta-\lambda_{n})\|\left(H+M^{*}\right)\tilde{z}_{n}\|_{S^{-1}}^{-2}\|\tilde{z}_{n}\|_{P}^{4} \label{eqnn:-14}\\
 & =\|z_{n}-z^{\star}\|_{S}^{2}-\lambda_{n}(\delta-\lambda_{n})\|P^{-1/2}S^{-1/2}\left(H+M^{*}\right)\tilde{z}_{n}\|_{P}^{-2}\|\tilde{z}_{n}\|_{P}^{4}\nonumber \\
 & \leq\|z_{n}-z^{\star}\|_{S}^{2}-\lambda_{n}(\delta-\lambda_{n})\|P^{-1/2}S^{-1/2}\left(H+M^{*}\right)\|_{P}^{-2}\|\tilde{z}_{n}\|_{P}^{2}\nonumber \\
 & =\|z_{n}-z^{\star}\|_{S}^{2}-\lambda_{n}(\delta-\lambda_{n})\|S^{-1/2}\left(H+M^{*}\right)P^{-1/2}\|^{-2}\|\tilde{z}_{n}\|_{P}^{2}. \label{eq:-12-23}
\end{align} 
Furthermore, when $C\equiv 0$ all the above analysis holds with $\delta=2$.
 
\eqref{thm-part1}: Inequality \eqref{eq:-12-23} and $(\lambda_{n})_{n\in\Nn}\subseteq[0,\delta]$ show that $(z_{n})_{n\in\Nn}$
is F\'{e}jer monotone with respect to $\zer (T)$ in the Hilbert space $\HH_{S}$. 

\eqref{thm-part2}: From \eqref{eq:-12-23}
and $\underset{n\rightarrow\infty}{\lim\inf} \,\lambda_{n}(\delta-\lambda_{n})>0$, it follows that $\tilde{z}_{n}\to 0$. 

\eqref{thm-part3}: Define 
\begin{equation}
w_{n}\coloneqq -(H-M)\tilde{z}_{n}+ C\bar{z}_n-C z_n.\label{eq:-2}
\end{equation}
It follows from \eqref{eq:-2}, linearity of $H-M$, cocoercivity of $C$ and \eqref{thm-part2} that
\begin{equation}
w_{n}\to 0.\label{eq:-4}
\end{equation}
By \hyperlink{alg:line1}{step 1} of \Cref{Algorithm-1} we have $(H-M-C)z_n\in (H+A)\bar{z}_n $, which together with \eqref{eq:-2} yields
\begin{equation}
w_{n}\in T\bar{z}_{n}.\label{eq:-3}
\end{equation}
Now let $z$ be a weak sequential cluster point of $(z_{n})_{n\in\Nn}$
, say $z_{k_{n}}\rightharpoonup z$. It follows from \eqref{thm-part2} that
$\bar{z}_{k_{n}}\rightharpoonup z$, and from \eqref{eq:-4}
that $w_{k_{n}}\to 0$. Altogether, by \eqref{eq:-3}, the members of the sequence $(\bar{z}_{k_{n}},w_{k_{n}})_{n\in\Nn}$
belong  to $\gra(T)$. Additionally, by \cite[Example 20.28, 20.29 and Corollary 24.4(i)]{bauschke2011convex}, $T$ is maximally monotone. Then, an appeal to \cite[Proposition
20.33(ii)]{bauschke2011convex} yields $(z,0)\in\gra(T)$. This together with~\eqref{thm-part1} and \cite[Theorem
5.5]{bauschke2011convex} completes the proof. 
\end{proof}

Equation~\eqref{eq:-12-23} implies that the sequence 
${(\min_{{\scriptstyle i=1\ldots n}}
	\| \tilde{z}_{i}\|_{P}^{2})_{n\in\Nn}}$,
the cumulative minimum of $(\|\tilde{z}_{n}\|_{P}^{2})_{n\in\Nn}$, 
converges sublinearly. Our next goal is to derive big-$O(1/(n+1))$ and little-$o(1/(n+1))$ convergence rates for the sequence itself. This is established below, under further restrictions on $\left(\lambda_n\right)_{n\in\Nn}$, by showing that the sequence $\left(\|\tilde{z}_n\|_D^2\right)_{n\in\Nn}$ is monotonically nonincreasing and summable.
\begin{thm}[Convergence rates] \label{Thm: conv-rates}
	Consider \Cref{Algorithm-1} under the assumptions of~\Cref{thm:Suppose-that-}. Let $c_1$ and $c_2$ be two positive constants satisfying 
		\begin{equation} \label{eqnn:-12}
		c_1 P \preceq D \preceq c_2 P,
		\end{equation}
		with $D$ defined in~\eqref{eq:D}, and assume 
		\begin{equation}\label{eqnn:-10}
		(\lambda_{n})_{n\in\Nn}\subseteq\left[0,c_1\delta/c_2\right],
		\end{equation} 
		 where $\delta$ is defined in~\eqref{eq:-7}. Then
		\begin{enumerate}[{label=(\textit{\roman*})},ref=\textit{\roman{*}}]
			\item \label{thm2-part1}$(\|\tilde{z}_n\|_D^2)_{n\in\Nn}$ is monotonically nonincreasing.
			\item \label{thm2-part2} If in addition to~\eqref{eq:-7} and~\eqref{eqnn:-10}, 
			$(\lambda_n(\delta-\lambda_n))_{n\in\Nn}\subseteq[\tau,\infty[$ for some $\tau>0$, then 
			\begin{equation*} 
			\|\tilde{z}_{n}\|_{D}^{2}\leq\frac{c_2^2}{\tau(n+1)}\|z_{0}-z^{\star}\|_{S}^{2}\quad\textrm{and}\quad \|\tilde{z}_{n}\|_{D}^{2}=o(1/(n+1)).
			\end{equation*}
		\end{enumerate}
	Furthermore, when $C\equiv 0$  all of the above statements hold with $\delta=2$.
\end{thm}
\begin{proof}
 \eqref{thm2-part1}: Using the monotonicity of $A$ and \hyperlink{alg:line1}{Step 1} of~\Cref{Algorithm-1} 
 \begin{align} \label{eqnn:-1}
 0 & \leq  \langle(H-M)(z_{n}-z_{n+1})-H(\bar{z}_{n}-\bar{z}_{n+1})+Cz_{n+1}-Cz_n,\bar{z}_{n}-\bar{z}_{n+1}\rangle.
 \end{align}
 On the other hand we have 
 \small
 \begin{align*}
 \langle Cz_{n+1}-Cz_{n},\bar{z}_{n}-\bar{z}_{n+1}\rangle =&\langle Cz_{n+1}-Cz_{n},\tilde{z}_{n}-\tilde{z}_{n+1}\rangle+ \langle Cz_{n+1}-Cz_{n},{z}_{n}-{z}_{n+1}\rangle \\
 \leq&\frac{\epsilon}{2}\|\tilde{z}_{n}-\tilde{z}_{n+1}\|_{P}^{2}+\frac{1}{2\epsilon}\|Cz_{n+1}-Cz_{n}\|_{P^{-1}}^{2}\\
 +&\langle Cz_{n+1}-Cz_{n},{z}_{n}-{z}_{n+1}\rangle\\
 \leq&\frac{\epsilon}{2}\|\tilde{z}_{n}-\tilde{z}_{n+1}\|_{P}^{2}+\left(1-\frac{1}{2\epsilon\beta}\right)\langle Cz_{n+1}-Cz_{n},{z}_{n}-{z}_{n+1}\rangle.
 \end{align*}\normalsize
 The first inequality follows from the Fenchel-Young inequality for $\frac{\epsilon}{2}\|\cdot\|_P^2$, and the second inequality follows from $\beta$-cocoercivity
 of $C$ with respect to $\|\cdot\|_P$. Set $\epsilon=\frac{1}{2\beta}$ so that
 \begin{equation}
 \langle Cz_{n+1}-Cz_{n},\bar{z}_{n}-\bar{z}_{n+1}\rangle  \leq\frac{1}{4\beta}\|\tilde{z}_{n}-\tilde{z}_{n+1}\|_{P}^{2}.\label{eqn:-4-2}
 \end{equation}
 Using~\eqref{eqnn:-1},~\eqref{eqn:-4-2} and monotonicity of $M$ we have
 \begin{align}
 0 
 & \leq \tfrac{1}{4\beta}\|\tilde{z}_{n}-\tilde{z}_{n+1}\|_{P}^{2}+\langle-M(z_{n}-z_{n+1})-H(\tilde{z}_{n}-\tilde{z}_{n+1}),\bar{z}_{n}-\bar{z}_{n+1}\rangle \nonumber\\
 & = \tfrac{1}{4\beta}\|\tilde{z}_{n}-\tilde{z}_{n+1}\|_{P}^{2} + \langle-M(z_{n}-z_{n+1})-H(\tilde{z}_{n}-\tilde{z}_{n+1}),\tilde{z}_{n}-\tilde{z}_{n+1}\rangle \nonumber\\
 & +\langle-M(z_{n}-z_{n+1})-H(\tilde{z}_{n}-\tilde{z}_{n+1}),z_{n}-z_{n+1}\rangle\nonumber\\
 & \leq \tfrac{1}{4\beta}\|\tilde{z}_{n}-\tilde{z}_{n+1}\|_{P}^{2} + \langle-M(z_{n}-z_{n+1})-H(\tilde{z}_{n}-\tilde{z}_{n+1}),\tilde{z}_{n}-\tilde{z}_{n+1}\rangle \nonumber\\
 &+ \langle-H(\tilde{z}_{n}-\tilde{z}_{n+1}),z_{n}-z_{n+1}\rangle \label{eqnn:-20}\\
 & =  -\left(1-\tfrac{1}{4\beta}\right)\|\tilde{z}_{n}-\tilde{z}_{n+1}\|_{P}^{2} -\langle(M+H^*)(z_{n}-z_{n+1}),\tilde{z}_{n}-\tilde{z}_{n+1}\rangle. \label{eqnn:-2}
 \end{align}
 It follows from~\eqref{eqnn:-2} and \hyperlink{alg:line4}{Step 4} of~\Cref{Algorithm-1} that  
 \begin{align}
 \left(1-\tfrac{1}{4\beta}\right)\|\tilde{z}_{n}-\tilde{z}_{n+1}\|_{P}^{2} & \leq \langle-(M+H^*)(z_{n}-z_{n+1}),\tilde{z}_{n}-\tilde{z}_{n+1}\rangle\nonumber \\
 & = \langle \alpha_n(H+M^*)^*S^{-1}(H+M^*)\tilde{z}_{n},\tilde{z}_{n}-\tilde{z}_{n+1}\rangle\nonumber\\
 &=\langle \alpha_nD\tilde{z}_{n},\tilde{z}_{n}-\tilde{z}_{n+1}\rangle. \label{eqnn:-42}
 \end{align}
 Let us show that $(\|\tilde{z}_n\|_D^2)_{n\in\Nn}$ is monotonically nonincreasing. Using the identity 
 \begin{equation} \label{eq:identity1}
 \|a\|_{D}^{2}-\|b\|_{D}^{2}=2\langle Da,a-b\rangle-\|a-b\|_{D}^{2},
 \end{equation}
 we have  	 
 \begin{align*}
 \|\tilde{z}_{n}\|_{D}^{2}-\|\tilde{z}_{n+1}\|_{D}^{2}&=2\langle D\tilde{z}_{n},\tilde{z}_{n}-\tilde{z}_{n+1}\rangle-\|\tilde{z}_{n}-\tilde{z}_{n+1}\|_{D}^{2} \nonumber\\
 & \geq 
 \tfrac{2}{\alpha_n}(1-\tfrac{1}{4\beta})\|\tilde{z}_{n}-\tilde{z}_{n+1}\|_{P}^{2}-\|\tilde{z}_{n}-\tilde{z}_{n+1}\|_{D}^{2}\nonumber\\
 & \geq \tfrac{c_1\delta}{\lambda_n}\|\tilde{z}_{n}-\tilde{z}_{n+1}\|_{P}^{2}-\|\tilde{z}_{n}-\tilde{z}_{n+1}\|_{D}^{2} \nonumber\\
 & \geq \left(\tfrac{c_1\delta}{c_2\lambda_n}-1\right)\|\tilde{z}_{n}-\tilde{z}_{n+1}\|_{D}^{2} 
 \end{align*}
 where the inequalities follow from~\eqref{eqnn:-42}, the definition of $\alpha_n$  and~\eqref{eqnn:-12}. The assertion follows from~\eqref{eqnn:-10} and the above inequality. 
 
 \eqref{thm2-part2}:  It follows from~\eqref{eqnn:-14} and~\eqref{eqnn:-12} that
 \begin{align}
 \|z_{n+1}-z^{\star}\|_{S}^{2} & \leq \|z_{n}-z^{\star}\|_{S}^{2}-\lambda_{n}(\delta-\lambda_{n})\|\left(H+M^{*}\right)\tilde{z}_{n}\|_{S^{-1}}^{-2}\|\tilde{z}_{n}\|_{P}^{4}\nonumber \\
 & = \|z_{n}-z^{\star}\|_{S}^{2}-\lambda_{n}(\delta-\lambda_{n})\|\tilde{z}_{n}\|_D^{-2}\|\tilde{z}_{n}\|_{P}^{4}\nonumber \\
 & \leq \|z_{n}-z^{\star}\|_{S}^{2}-c_2^{-2}\lambda_{n}(\delta-\lambda_{n})\|\tilde{z}_n\|_D^2.\nonumber
 \end{align}
 Summing over $n$ yields $\textstyle{\sum_{i=0}^{\infty}} \lambda_i(\delta-\lambda_i)\|\tilde{z}_i\|_D^2 \leq c_2^{2}\|z_{0}-z^{\star}\|_{S}^{2}$. 
 It then follows from $(\lambda_n(\delta-\lambda_n))_{n\in\Nn}\subseteq[\tau,\infty[$ that
 \begin{equation} \label{eqnn:-8}
 \textstyle{\sum_{i=0}^{\infty}} \|\tilde{z}_i\|_D^2 \leq \frac{c_2^{2}}{\tau}\|z_{0}-z^{\star}\|_{S}^{2}. 
 \end{equation} 
 On the other hand \eqref{thm2-part1} yields 
 \begin{equation} \label{eqnn:-7}
 \|\tilde{z}_{n}\|_{D}^{2}\leq\frac{1}{n+1}\sum_{i=0}^{n}\|\tilde{z}_{i}\|_{D}^{2}.
 \end{equation}
 Combining~\eqref{eqnn:-8} and \eqref{eqnn:-7} establishes the big-$O$ convergence.
 The little-$o$ convergence follows from \eqref{thm2-part1},~\eqref{eqnn:-8} and \cite[Lemma 3-(1a)]{davis2014convergence}.
	\end{proof}
We can show stronger convergence results, \ie, linear convergence rate, under \emph{metric subregulariy} assumption for $T$. We restate the following definition from \cite{dontchev2004regularity}:

\begin{deff}[Metric subregularity] \label{metricsub}
A mapping $F$ is metrically subregular at $\bar{x}$ for $\bar{y}$ if $(\bar{x},\bar{y})\in\gra F$ and there exists $\eta\in[0,\infty[$, 
a neighborhood $\mathcal{U}$ of $\bar{x}$ and $\mathcal{V}$ of $\bar{y}$ such that 
\begin{equation}\label{eqnn:metricsubregularity}
d(x,F^{-1}\bar{y})\leq \eta d(\bar{y},Fx\cap \mathcal{V})  \;\; \textrm{for all} \; x\in \mathcal{U}. 
\end{equation}
\end{deff}
This is equivalent to \emph{calmness} of the operator $F^{-1}$ at $\bar{y}$ for $\bar{x}$ \cite[Theorem 3.2]{dontchev2004regularity}. The above two properties are weaker versions of metric regularity and Aubin property, respectively. We refer the reader to \cite[Chapter 9]{rockafellar2009variational} and \cite[Chapter 3]{dontchev2009implicit} for an extensive discussion. In \Cref{Thm: conv-rates-ii}, we derive linear convergence rates when the operator $T=A+M+C$ is {metrically subregular} at all $z^\star\in\zer (T)$ for $0$. Metric subregularity is used in \cite{liang2014convergence} to show linear convergence of Krasnosel'ski\v{\i}-Mann iterations for finding a fixed point of a nonexpansive mapping.
\begin{thm}[Linear convergence] \label{Thm: conv-rates-ii}
Consider \Cref{Algorithm-1} under the assumptions of~\Cref{thm:Suppose-that-}.  Suppose that $T$ is {metrically subregular} at all $z^\star\in\zer (T)$ for $0$,  \emph{cf.}~\eqref{eqnn:metricsubregularity}.
  If either $\HH$ is finite-dimensional or $\mathcal{U}=\HH$, 
  then $(d_S({z}_{n},\zer(T)))_{n\in\Nn}$ converges $Q$-linearly to zero,  $(z_n)_{n\in\Nn}$ and $(\|\tilde{z}_{n}\|_P)_{n\in\Nn}$ converge $R$-linearly to some $z^\star\in\zer(T)$ and zero, respectively.\footnote{The sequence $(x_{n})_{n\in\Nn}$ converges to $x^{\star}$ $R$-linearly if there is a sequence of nonnegative scalars $({v_n})_{n\in\Nn}$ such that  $\|x_{n}-x^{\star}\|\leq v_n$ and $(v_n)_{n\in\Nn}$ converges $Q$-linearly\footnotemark{} to zero.} \footnotetext{The sequence $(x_{n})_{n\in\Nn}$ converges to $x^{\star}$ $Q$-linearly with $Q$-factor given by $\sigma\in]0,1[$, if for $n$ sufficiently large  $\|x_{n+1}-x^{\star}\|\leq\sigma\|x_{n}-x^{\star}\|$ holds.}
Furthermore, when $C\equiv 0$  the above statements hold with $\delta=2$.
\end{thm}
\begin{proof}
It follows from metric subregularity of $T$ at all $z^\star\in\zer (T)$ for $0$ that 
\begin{equation} \label{eqnn:calmness-equiv}
d(x,\zer (T))\leq\eta\|y\|\;\; \forall x\in \mathcal{U} \,\;\textrm{and} \,\; y\in Tx \,\;\textrm{with}\;\|y\|\leq\nu,
\end{equation}
for some $\nu\in]0,\infty[$ and $\eta\in[0,\infty[$ and a neighborhood $\mathcal{U}$ of $\zer (T)$. Consider $w_{n}$ defined in \eqref{eq:-2}. It was shown in~\eqref{eq:-4} that $w_{n}\to 0$ and if $\HH$ is a finite-dimensional Hilbert space, \Cref{thm:Suppose-that-}\eqref{thm-part2}-\eqref{thm-part3} yield that $\bar{z}_n$ converges to a point in $\zer (T)$. Then there exists
$\bar{n}\in\Nn$ such that for $n>\bar{n}$ we have $\|w_{n}\|\leq\nu$ and a neighborhood $\mathcal{U}$ of $\zer (T)$ exists with $\bar{z}_n\in \mathcal{U}$ (This holds trivially when $\mathcal{U}=\HH$). Consequently~\eqref{eqnn:calmness-equiv} yields
$d(\bar{z}_{n},\zer (T))\leq\eta\|w_{n}\|$. In addition, triangle inequality and Lipschitz
continuity of $C$ yield 
\begin{align}
	\|w_{n}\| & =\|H\tilde{z}_{n}-M\tilde{z}_{n}-C\bar{z}_{n}+C{z}_{n}\|\leq \|(H-M)\tilde{z}_{n}\|+\|C\bar{z}_{n}-Cz_{n}\|\nonumber \\ 
	& \leq\left(\|H-M\|+\tfrac{1}{\beta}\|P\|\right)\|\tilde{z}_{n}\|.\nonumber
\end{align}
Consider the projection of $\bar{z}_n$ onto $\zer(T)$, $\Pi_{\zer(T)}(\bar{z}_n)$. By definition  $\|\bar{z}_{n}-\Pi_{\zer(T)}(\bar{z}_n)\|=d(\bar{z}_{n},\zer (T))$ (the minimum is attained since $T$ is maximally monotone \cite[Proposition 23.39]{bauschke2011convex}), and we have
\begin{align}
	\|z_{n}-\Pi_{\zer(T)}(\bar{z}_n)\| & \leq  \|\bar{z}_{n}-\Pi_{\zer(T)}(\bar{z}_n)\|+\|\tilde{z}_{n}\|  = d(\bar{z}_{n},\zer (T))+\|\tilde{z}_{n}\| \nonumber \\
	& \leq \xi\eta\|\tilde{z}_{n}\| +\|\tilde{z}_{n}\|  \leq  (\xi\eta+1)\|P^{-1}\|^{1/2}\|\tilde{z}_{n}\|_{P}, \label{eq:-40}
\end{align}
where $\xi=\|H-M\|+\frac{1}{\beta}\|P\|$. 
It follows from~\eqref{eq:-40} that 
\begin{align} \label{eq:-110}
	d_S^2(z_n,\zer(T))  \leq \|z_n-\Pi_{\zer(T)}(\bar{z}_n)\|_S^2 \leq (\xi\eta+1)^2\|P^{-1}\|\|S\|\|\tilde{z}_{n}\|^2_{P}.
\end{align}
For $\Pi^S_{\zer(T)}(z_n)$ by its definition we have $\|{z}_{n}-\Pi^S_{\zer(T)}({z}_n)\|_S=d_S({z}_{n},\zer(T))$, and since inequality~\eqref{eqnn:-14} holds for all $z^\star\in\zer(T)$, it follows that 
\small
\mathtight{ \begin{align}
		d^2_S({z}_{n+1},\zer(T)) & \leq \|z_{n+1}-\Pi^S_{\zer(T)}(z_n)\|^2_S \nonumber \\
		& \leq \|z_{n}-\Pi^S_{\zer(T)}(z_n)\|_{S}^{2}-\lambda_{n}(\delta-\lambda_{n})\|\left(H+M^{*}\right)\tilde{z}_{n}\|_{S^{-1}}^{-2}\|\tilde{z}_{n}\|_{P}^{4} \nonumber \\
		& = d^2_S({z}_{n},\zer(T)) -\lambda_{n}(\delta-\lambda_{n})\|\left(H+M^{*}\right)\tilde{z}_{n}\|_{S^{-1}}^{-2}\|\tilde{z}_{n}\|_{P}^{4} \label{eqnn:zn}\\
		& \leq d^2_S({z}_{n},\zer(T))-\lambda_{n}(\delta-\lambda_{n})\|S^{-1/2}\left(H+M^{*}\right)P^{-1/2}\|^{-2}\|\tilde{z}_{n}\|_{P}^{2} \label{eqnn:tildez}\\
		&\leq  d^2_S({z}_{n},\zer(T))-\tfrac{\lambda_{n}(\delta-\lambda_{n})}{(\xi\eta+1)^2\|P^{-1}\|\|S\|\|S^{-1/2}\left(H+M^{*}\right)P^{-1/2}\|^{2}}d^2_S({z}_{n},\zer(T)), \nonumber
	\end{align}}\normalsize
	where in the last inequality we used~\eqref{eq:-110}. It follows from~\eqref{eq:-7} that there exists $\tilde{n}\in\Nn$ such that $\left(\lambda_{n}(\delta-\lambda_{n})\right)_{n>\tilde{n}}\subseteq[\bar{\tau},\infty[$ for some  $\bar{\tau}>0$. Therefore,  $\left(d_S({z}_{n},\zer(T))\right)_{n\in\Nn}$ converges $Q$-linearly to zero. $R$-linear convergence of  $\left(\|\tilde{z}_{n}\|_{P}\right)_{n\in\Nn}$ follows from~\eqref{eqnn:tildez} and $Q$-linear convergence of $\left(d_S({z}_{n},\zer(T))\right)_{n\in\Nn}$. \hyperlink{alg:line4}{Step 4} of \Cref{Algorithm-1} and~\eqref{eqnn:zn} yield 
	\small
	\begin{align*}
		\|z_{n+1}-z_n\|_S^2 =  \lambda_{n}^{2}\|(H+M^{*})\tilde{z}_{n}\|_{S^{-1}}^{-2}\|\tilde{z}_{n}\|_{P}^{4}
		\leq  \frac{\delta^2}{\bar{\tau}}\left(d^2_S({z}_{n},\zer(T))-d^2_S({z}_{n+1},\zer(T))\right). 
	\end{align*}\normalsize
	Therefore,  $\left(\|z_{n+1}-z_n\|_S\right)_{n\in\Nn}$ converges $R$-linearly to zero. This is equivalent to saying that there exists $c\in]0,1[$, $\kappa\in]0,\infty[$, $\underline{n}\in\Nn$ such that for all $n\geq\underline{n}$, $\|z_{n+1}-z_n\|_S\leq \kappa c^n$ holds. 
	Thus, for any $j>k\geq\underline n$ we have 
	\begin{equation}\label{eq:R-linear}
	\|z_j-z_k\|_S  \leq \textstyle{\sum_{i=k}^{j-1}}\|z_{i+1}-z_i\|_S\leq \sum_{i=k}^{j-1}\kappa c^i\leq \sum_{i=k}^{\infty}\kappa c^i=\frac{\kappa}{1-c}c^k.
	\end{equation}
	Hence, the sequence $(z_n)_{n\in\Nn}$ is a Cauchy sequence, and therefore converges to some $z\in\HH$. From uniqueness of weak limit and \Cref{thm:Suppose-that-}\eqref{thm-part3} we have $z\in\zer(T)$. Let $j\rightarrow\infty$ in~\eqref{eq:R-linear} to obtain $R$-linear convergence of $(z_n)_{n\in\Nn}$. 
\end{proof}
In the special case when $C\equiv0$, $M$
is skew-adjoint , $K=M$ and $S=P$, the operator $P\in\mathscr{B}(\HH)$ can be a self-adjoint,
positive operator rather than a strongly positive operator. Under these assumptions AFBA simplifies to the following iteration:
\begin{subequations}
\begin{align}
\bar{z}_{n} & =(H+A)^{-1}Pz_{n} \label{Algorithm-Pi}\\
z_{n+1} & = z_n + \lambda_n (\bar{z}_{n}-z_n). \label{Algorithm-Pii}
\end{align}
\end{subequations}
Notice that if $P$ was strongly positive, this could simply be seen as proximal point algorithm in a different metric applied to the operator $A+M$, but we have relaxed this assumption and only require $P$ to be positive. Before providing convergence results for this algorithm we begin with the following lemma, showing that the mapping $(H+A)^{-1}$ has full domain and is continuous when $H$ has a block triangular structure with strongly positive diagonal blocks, even though its symmetric part, $P$, might not be strongly positive. This lemma motivates the assumption on continuity of $(H+A)^{-1}P$ in 
\Cref{thm:Suppose-that--1}. As an application of this theorem in \Cref{prop:-1}\eqref{prop:-1-part3}, when $P$ is  positive with a two-by-two block structure (see~\eqref{eq:-42-p} in the limiting case $\theta=2$), DRS is recovered. 

\begin{lem} \label{lem:separable}
	Let $\HH=\HH_1\oplus\cdots\oplus\HH_N$, where  $\HH_1,\cdots,\HH_N$ are real Hilbert spaces. 
	Suppose that $A$ is block separable and $H$ has a conformable lower (upper) triangular partitioning, \ie,   
	\[
	A:z\mapsto(A_1z_1,\cdots,A_Nz_N),
	\]
	\begin{equation} \label{eq:blclowerH}
	H:z\mapsto(H_{11}z_1,H_{21}z_1+H_{22}z_2,\cdots, \sum_{j=1}^{N}H_{Nj}z_j),
	\end{equation}    
	where $z_i\in\HH_i$ for $i=1,\cdots,N$,  $z=(z_1,\cdots,z_N)\in\HH$, and $H_{ij}\in\mathscr{B}(\HH_j,\HH_i)$ for $i,j=1,\cdots,N$. For $i=1,\cdots,N$, assume that $A_i$ is maximally monotone, and  $H_{ii}\in\mathcal{S}_{\tau_i}(\HH_i)$ with $\tau_i\in]0,\infty]$. Then, the mapping $(H+A)^{-1}$ is continuous and has full domain, \ie, $\dom((H+A)^{-1})=\HH$. Furthermore, the update $\bar{z}=(H+A)^{-1}z$ is carried out using
	\begin{equation}\label{eq:resolventform}
	\bar{z}_i=\begin{cases}
	\left(H_{11}+A_1\right)^{-1}z_1, &  \quad i=1;\\
	\left(H_{ii}+A_i\right)^{-1}\left(z_i-\sum_{j=1}^{i-1}H_{ij}\bar{z}_j\right), & \quad i=2,\cdots,N,
	\end{cases}
	\end{equation}
	where $\bar{z}=(\bar{z}_1,\cdots,\bar{z}_N)\in\HH$.
\end{lem}
\begin{proof}
		We consider a block lower triangular $H$ as in~\eqref{eq:blclowerH}, the analysis for upper triangular case is identical. The goal is to consider $A_i$'s separately. 	Let $\bar{z}=(H+A)^{-1}z$ with $\bar{z}=(\bar{z}_1,\cdots,\bar{z}_N)$.  The block triangular structure of $H$ in~\eqref{eq:blclowerH} yields the  equivalent inclusion $z_i\in A_i\bar{z}_i+ \sum_{j=1}^{i}H_{ij}\bar{z}_j$, for $i=1,\cdots,N$. 
		This is equivalent to~\eqref{eq:resolventform}, in which, each $\bar{z}_i$ is evaluated using $z_i$ and $\bar{z}_j$ for $j<i$. 
		For the first block we have $\bar{z}_1=(H_{11}+A_1)^{-1}z_1$. 	Since $A_1$ is monotone and $H_{11}$ is strongly monotone, it follows that $H_{11}+A_1$ is strongly monotone, which in turn implies that $(H_{11}+A_1)^{-1}$ is cocoercive and, as such, at most single-valued and continuous. Since $A_1$ is maximally monotone and $H_{11}$ is strongly positive we have \[\dom\left((H_{11}+A_1)^{-1}\right)=\ran(H_{11}+A_1)=\ran(\id+A_1H_{11}^{-1})=\HH_1,\] where the last equality follows from maximal monotonicity of $A_1 H_{11}^{-1}$ in the Hilbert space defined by endowing $\HH_1$ with the scalar product $\langle \cdot,\cdot\rangle_{H_{11}^{-1}}$, and
		Minty's theorem \cite[Theorem 21.1]{bauschke2011convex}. 
		For the second block in~\eqref{eq:resolventform} we have 
		$\bar{z}_2= (H_{22}+A_2)^{-1}(z_2-H_{21}\bar{z}_1)$.
		Hence, by the same argument used for previous block, $(H_{22}+A_2)^{-1}$ is continuous and has full domain. Since $(z_2-H_{21}\bar{z}_1)$ and $(H_{22}+A_2)^{-1}$ are continuous, so is their composition $(H_{22}+A_2)^{-1}(z_2-H_{21}\bar{z}_1)$.
		Follow the same argument for the remaining blocks in~\eqref{eq:resolventform} to conclude that $(H+A)^{-1}$ is continuous and has full domain. 
\end{proof}
Next theorem provides convergence and rate of convergence results for algorithm \eqref{Algorithm-Pi}-\eqref{Algorithm-Pii} in finite-dimensions by employing the same idea used in \cite[Theorem 3.3]{condat2013primal}. The idea is to consider the operator $R=P+\id-Q$, where $Q$ is the orthogonal projection onto $\ran(P)$. The proof presented here is for a general $P$ and it coincides with the one of Condat \cite[Theorem 3.3]{condat2013primal} for a special choice of $P$ (when $P$ is defined as in~\eqref{eq:-17-1} with $\theta=2$).   

\begin{thm}
	\label{thm:Suppose-that--1} Suppose that $\HH$ is finite-dimensional. Let $P\in \mathcal{S}(\HH)$, $P\succeq0$,  $M\in\mathscr{B}(\HH)$ a
	skew-adjoint operator and $H=P+M$. Consider the iteration \eqref{Algorithm-Pi}-\eqref{Algorithm-Pii} and assume $\zer (T)\neq \textrm{\O}$ where $T=A+M$. Furthermore, assume that $(H+A)^{-1}P$ is continuous. Let $(\lambda_{n})_{n\in\Nn}$ be uniformly bounded in the interval $]0,2[$. Then,
	\begin{enumerate}[{label=(\textit{\roman*})},ref=\textit{\roman{*}}]
		\item \label{thm-P-part3} $(z_{n})_{n\in\Nn}$ converges to a point in $\zer (T)$.
		\item \label{thm-P-part4} Let $Q$ be the orthogonal projection onto $\ran(P)$, and $R=P+\id-Q$. 
		The following convergence estimates hold:
		\begin{equation} \label{positiv-P-rate} 
			\|{P}z_{n+1}-Pz_n\|^{2}\leq\tfrac{\|P\|}{\tau(n+1)}\|Qz_{0}-Qz^{\star}\|_{R}^{2},
		\end{equation}
		for some constant $\tau>0$, and $\|Pz_{n+1}-Pz_n\|^{2}=o(1/(n+1))$.
	\end{enumerate}
\end{thm}

\begin{proof}
	\eqref{thm-P-part3}: Since $P$ is not strongly positive, it does not define a valid inner product. Consider $R\coloneqq P+\id-Q$, where $Q$ is the
	orthogonal projection onto $\ran(P)$. We show that by construction $R$ is strongly positive.  
	By the spectral theorem we can write 
	$Pz_1=U\Lambda U^*z_1$, where $U$ is an orthonormal basis consisting of eigenvectors of $A$. Consider two sets: $s_1=\{i|\lambda_i\neq0\}$ and $s_2=\{i|\lambda_i=0\}$. 
	Denote by $U_1$ the orthonormal basis made up of $u_i$ for $i\in s_1$. Then $\ran{(P)}=\ran{(U_1)}$ and we have $Q=U_1U_1^*$. For any $z\in \HH$, $z=z_{1}+z_{2}$ where $z_{1}=Qz$ and $z_{2}=(\id-Q)z$. Then $Rz=Pz+z-Qz=Pz_{1}+z_{2}$ and $\langle Rz,z\rangle =\langle Pz_{1}+z_{2},z\rangle =\langle Pz_{1},z_{1}\rangle+\|z_{2}\|^{2}$.
	If $z_1=0$ then $\langle Rz,z\rangle=\|z_{2}\|^{2}=\|z\|^{2}$. 
	Suppose that $z_2\neq0$ and $z_1\neq 0$. Denote by $\lambda_{\min}$ the smallest non zero eigenvalue of $P$. We have
	\begin{align*} 
		\langle Rz,z\rangle & =\langle Pz_{1},z_{1}\rangle+\|z_2\|^2  = \langle U\Lambda U^* z_{1},z_{1}\rangle+\|z_2\|^2 = \langle \Lambda U^* z_{1},U^* z_{1}\rangle+\|z_2\|^2\\
		& \geq \lambda_{\min}\|U_1^*z_1\|^2+\|z_2\|^2= \lambda_{\min} \langle z_1,Qz_1 \rangle+\|z_2\|^2= \lambda_{\min}\|z_1\|^2+\|z_2\|^2 \\
		&\geq \min\{1,\lambda_{\min}\}\|z\|^{2}.
	\end{align*} 
	If $z_2=0$ the above analysis holds with $z=z_1$ and result in strong positivity parameter equal to $\lambda_{\min}$. 
	
	We continue by noting that by definition we have $Q\circ P=P$, and symmetry of $P$ yields $P\circ Q=P$. Therefore, $R\circ Q=P$ and for $z\in\HH$ we have
	\begin{align}
		\langle Pz,z\rangle=  	\langle QPz,z\rangle  
		= 	\langle Pz,Qz\rangle 
		=  \langle RQz,Qz\rangle, \label{eq:RQ}
	\end{align} 
	which will be used throughout this proof. Observe now that for $z^{\star}\in\zer(T)\neq \emptyset$ we have $-Mz^{\star}\in A z^{\star}$. By monotonicity of ${A}$ and~\eqref{Algorithm-Pi} we have 
	$\langle M\bar{z}_{n}-Mz^{\star}+P\tilde{z}_{n},z^{\star}-\bar{z}_{n}\rangle \geq 0$.	Then
	\begin{align}
		0 & \leq \langle M\bar{z}_{n}-Mz^{\star}+P\tilde{z}_{n},z^{\star}-\bar{z}_{n}\rangle=\langle P\tilde{z}_{n},z^{\star}-\bar{z}_{n}\rangle \nonumber\\
		& =\langle P\tilde{z}_{n},Qz^{\star}-Q\bar{z}_{n}\rangle + \langle RQ\tilde{z}_{n},Q\tilde{z}_{n}\rangle-\langle RQ\tilde{z}_{n},Q\tilde{z}_{n}\rangle\nonumber\\
		& =\langle P\tilde{z}_{n},Qz^{\star}-Qz_{n}\rangle-\|Q\tilde{z}_{n}\|_{R}^{2}, \label{eq:pp2}
	\end{align}
	where the equalities follows from skew-symmetricity  of $M$ 
	and \eqref{eq:RQ}. We show that $\|Qz_{n}-Qz^{\star}\|_{R}$ is decreasing using \eqref{eq:pp2}:
	\begin{align}
		\|Qz_{n+1}-Qz^{\star}\|_{R}^{2} & =\|Qz_{n}-Qz^{\star}+\lambda_n Q\tilde{z}_{n}\|_{R}^{2} \nonumber\\
		& =\|Qz_{n}-Qz^{\star}\|_{R}^{2}+2\lambda_n \langle Qz_{n}-Qz^{\star},P\tilde{z}_{n}\rangle+\lambda_n^2 \|Q\tilde{z}_{n}\|_{R}^{2}\nonumber\\
		& \leq\|Qz_{n}-Qz^{\star}\|_{R}^{2}-\lambda_n(2-\lambda_n)\|Q\tilde{z}_{n}\|_{R}^{2}. \label{eq:fejer mon}
	\end{align}
	
	Let us define the sequence $x_{n}=Qz_n$ and $\bar{x}_n=Q\bar{z}_n$ for every $n\in\Nn$.  Then, since $P\circ Q=P$ the iteration for $x_n=Qz_n$ is written as
	\begin{align}
		\bar{x}_{n} & =Q(H+A)^{-1}P(x_{n})\nonumber\\
		x_{n+1} & = x_n + \lambda_n (\bar{x}_{n}-x_n)\label{eq1:-6}.
	\end{align}
	Let $G=(H+A)^{-1}P$ and $G^\prime=Q\circ G$. 
	It follows from $H=P+M$ and \eqref{Algorithm-Pi} that 
	\begin{equation} \label{eq1:-11}
	0  \in T(G(z)) + PG(z)-Pz. 
	\end{equation}
	Use~\eqref{eq1:-11} and monotonicity of $T$ at $z^\star\in\zer(T)$ and $G(z^\star)$ to derive 
	\begin{equation} \label{eq1:-4}
	0\leq\langle G(z^\star)-z^\star,Pz^\star-PG(z^\star)\rangle.
	\end{equation}
	In view of~\eqref{eq1:-4} 
	and positivity of $P$, we have $\langle G(z^\star)-z^\star,PG(z^\star)-Pz^\star\rangle = 0$, and by \cite[Corollary 18.17]{bauschke2011convex}, $PG(z^\star)-Pz^\star=0$. Hence, since $R\circ Q=P$, we have $RQG(z^\star)-RQz^\star=0$, and 
	strong positivity of $R$ implies  $Qz^\star=QG(z^\star)=QG(Qz^\star)$, where the last equality is due to $G\circ Q=G$. Thus, $Qz^\star$ is a fixed point of $G^\prime=QG$. We showed that if $z^\star\in\zer(T)$ then $Qz^\star\in\fix(G^\prime)$, \ie, 
	\begin{equation} \label{eq:subset1}
	Q\zer(T)\subseteq \fix(G^\prime).
	\end{equation}
	Furthermore, for any $x^\star\in\fix(G^\prime)$ we have $Px^\star=PG^\prime(x^\star)=PQG(x^\star)=PG(x^\star)$.
	Combine this with \eqref{eq1:-11} to derive $G(x^\star)\in\zer(T)$. Therefore, $x^\star=G^\prime(x^\star)=QG(x^\star)\in Q\zer(T)$. This shows that if  $x^\star\in\fix(G^\prime)$, then  $x^\star\in Q\zer(T)$, \ie, $\fix(G^\prime)\subseteq Q\zer(T)$. Combine this with~\eqref{eq:subset1} to conclude that the two sets $\fix(G^\prime)$ and $Q\zer(T)$ are the same.  
	On the other hand, we rewrite~\eqref{eq:fejer mon} for $x_n$ and $\bar{x}_n$: 
	\begin{align}
		\|x_{n+1}-Qz^{\star}\|_{R}^{2}  \leq\|x_{n}-Qz^{\star}\|_{R}^{2}-\lambda_n(2-\lambda_n)\|\bar{x}_{n}-x_n\|_{R}^{2}. \label{eq:x fejer mon}
	\end{align}
	Therefore, $(x_{n})_{n\in\Nn}$ is F\'{e}jer monotone with respect to $\fix(G^\prime)$ in the Hilbert space $\HH_R$. Since  $(\lambda_{n})_{n\in\Nn}$ is uniformly bounded in $]0,2[$, it follows that 
	\begin{equation}\label{eq:-122}
	G^\prime(x_n)-x_n=\bar{x}_n-x_n\to 0.
	\end{equation} 
	Let $x$ be a sequential cluster point of $(x_n)_{n\in\Nn}$, say $x_{k_n}\to x$. $G^\prime$ is continuous since $G^\prime=Q\circ G$ and $G$ is assumed to be continuous. Thus, it follows from~\eqref{eq:-122} that
	$G^\prime(x)-x=0$, \ie, $x\in \fix(G^\prime)$. This together with F\'{e}jer monotonicity of $x_n$ with respect to $\fix(G^\prime)$ and \cite[Theorem 5.5]{bauschke2011convex} yields  $x_n\to x\in\fix(G^\prime)$. 
	
	The proof is completed by first using $G\circ Q=G$ and continuity of $G$ to deduce that $\bar{z}_n=G(z_n)=G(x_n)$ converges to $G(x^\star)\in\zer(T)$, and then arguing for convergence of $z_n$. 
	We skip the details here because they are identical to the last part of the proof in \cite[Theorem 3.3]{condat2013primal}).
	
	\eqref{thm-P-part4}: Follow the procedure in the proof of \Cref{Thm: conv-rates} to derive~\eqref{eqnn:-20}, except that in this case the cocoercive term is absent. This yields
	\small
	\begin{equation} 
	0 \leq \langle-M(z_{n}-z_{n+1})-H(\tilde{z}_{n}-\tilde{z}_{n+1}),\tilde{z}_{n}-\tilde{z}_{n+1}\rangle 
	- \langle H(\tilde{z}_{n}-\tilde{z}_{n+1}),z_{n}-z_{n+1}\rangle.\label{eqnn:-21}
	\end{equation} \normalsize
	Since $H=P+M$ and $M$ is skew-symmetric,~\eqref{eqnn:-21} simplifies to
	\begin{align} 
		0 & \leq \langle-P(\tilde{z}_{n}-\tilde{z}_{n+1}),\tilde{z}_{n}-\tilde{z}_{n+1}\rangle + \langle -P(z_{n}-z_{n+1}),\tilde{z}_{n}-\tilde{z}_{n+1}\rangle \nonumber \\
		& = -\|Q\tilde{z}_{n}-Q\tilde{z}_{n+1}\|_R^2 + \lambda_n\langle P\tilde{z}_{n},\tilde{z}_{n}-\tilde{z}_{n+1}\rangle, \label{eq:-120}
	\end{align}
	where we used~\eqref{eq:RQ}  and~\eqref{Algorithm-Pii}. Using  identity~\eqref{eq:identity1}, we derive
	\begin{align}
		\|Q\tilde{z}_{n}\|_{R}^{2}-\|Q\tilde{z}_{n+1}\|_{R}^{2}&=2\langle RQ\tilde{z}_{n},Q\tilde{z}_{n}-Q\tilde{z}_{n+1}\rangle-\|Q\tilde{z}_{n}-Q\tilde{z}_{n+1}\|_{R}^{2} \nonumber\\
		&=2\langle P\tilde{z}_{n},\tilde{z}_{n}-\tilde{z}_{n+1}\rangle-\|Q\tilde{z}_{n}-Q\tilde{z}_{n+1}\|_{R}^{2} \nonumber\\
		& \geq 
		\left(\tfrac{2}{\lambda_n}-1\right)\|Q\tilde{z}_{n}-Q\tilde{z}_{n+1}\|_{R}^{2}, \label{eqnn:-22} 
	\end{align}
	where we made use of~\eqref{eq:-120}. 
	Consider~\eqref{eq:fejer mon} and sum over $n$ to derive
	\begin{equation}\label{eq1:-100}
	\textstyle{\sum_{i=0}^{\infty}} \lambda_i(2-\lambda_i)\|Q\tilde{z}_i\|_R^2\leq \|Qz_0-Qz^\star\|^2_R.
	\end{equation}
	Inequality~\eqref{eqnn:-22} shows that $\|Q\tilde{z}_{n}\|_{R}^{2}$ is monotonically nonincreasing.
	Combine this with~\eqref{eq1:-100} and uniform boundedness of $\lambda_n$, \ie,  $(\lambda_{n})_{n\in\Nn}\subseteq[\epsilon,2-\epsilon]$ for some $\epsilon>0$, to derive 
	\begin{equation}\label{eq1:-7} 
	\|Q\tilde{z}_{n}\|_{R}^{2}\leq\frac{1}{(n+1)\epsilon^2}\|Qz_{0}-Qz^{\star}\|_{R}^{2}.
	\end{equation}
	Furthermore, it follows from \eqref{eq1:-6} and definition of $x_n,\bar{x}_n$ that
	\begin{equation} \label{eqnn:Qzbar-x}
	\|x_{n+1}-x_n\|_R^2=\lambda_n^2 \|Q\tilde{z}_{n}\|^2_R\leq (2-\epsilon)^2\|Q\tilde{z}_{n}\|^2_R.
	\end{equation}
	Combine~\eqref{eqnn:Qzbar-x} and \eqref{eq1:-7} to derive
	\begin{equation} \label{eq:bigO-x} 
	\|x_{n}-x_{n+1}\|_{R}^{2}\leq\frac{(2-\epsilon)^2}{(n+1)\epsilon^2}\|Qz_{0}-Qz^{\star}\|_{R}^{2}.
	\end{equation}
	This establishes big-$O$ convergence for $(x_{n})_{n\in\Nn}$.
	The little-$o$ convergence of $\|Q\tilde{z}_{n}\|_{R}^{2}$ and subsequently $\|x_{n}-x_{n+1}\|_{R}^{2}$ follows from~\eqref{eqnn:-22},~\eqref{eq1:-100} and \cite[Lemma 3-(1a)]{davis2014convergence}. We derive from~\eqref{eq:RQ} that $\|x_{n}-x_{n+1}\|_{R}^{2}= \langle  z_{n}-z_{n+1},P(z_{n}-z_{n+1})\rangle$. 
	Then it follows from \cite[Corollary 18.17]{bauschke2011convex} that
	\begin{equation} \label{eq:-102}
	\|Pz_{n}-Pz_{n+1}\|^{2}\leq \|P\|\|x_{n}-x_{n+1}\|_{R}^{2}.
	\end{equation} 
	Set $\tau=\tfrac{\epsilon^2}{(2-\epsilon)^2}$, and combine~\eqref{eq:-102} with~\eqref{eq:bigO-x} to yield big-$O$ convergence for the sequence $(Pz_n)_{n\in\Nn}$. Similarly little-$o$ convergence follows from that property of $\|x_{n}-x_{n+1}\|_{R}^{2}$.
\end{proof}

\section{Operator Splitting Schemes as Special Cases} \label{sec:special cases}
We are ready to consider some important special cases to illustrate the importance of parameters $S$, $P$ and $K$. Further discussion on other special choices for the parameters appear in \Cref{sec:Saddle-point-problem} in the framework of convex optimization with the understanding that it is straightforward to adapt the same analysis for the corresponding monotone inclusion problem.
\subsection{Forward-Backward Splitting}
 When $H=\gamma^{-1} \id$, $S=\id$ and $M\equiv0$,  \Cref{Algorithm-1} reduces to FBS. Let $\beta$ be the cocoercivity constant of $C$ with respect to the canonical norm $\|\cdot\|$, then $\beta/\gamma$ is the cocoercivity constant with respect to the $P$ norm 
 and condition~\eqref{eq:-7} of \Cref{thm:Suppose-that-} becomes
\begin{equation}
(\lambda_{n})_{n\in\Nn}\subseteq[0,\delta] \quad \textrm{with} \quad \delta=2-\frac{\gamma}{2\beta},\quad\gamma\in]0,4\beta[,\quad\liminf_{n\to\infty} \,\lambda_{n}(\delta-\lambda_{n})>0. \label{forwardeq}
\end{equation} 
This allows a wider range of parameters than the standard ones found in the literature. The standard convergence results for FBS are based on the theory of averaged operators (see \cite{combettes2015compositions} and the references therein) and yield the same conditions as in \eqref{forwardeq} but with $\gamma\in]0, 2\beta]$ (see also \cite[Lemma 4.4]{condat2013primal} and  a slightly more conservative version in \cite[Theorem 25.8]{bauschke2011convex}). 
Additionally, if $C\equiv0$, FBS reduces to the classical PPA. 
 
The convergence rate for FBS follows directly from \Cref{Thm: conv-rates}. Since $D=\gamma^{-2}\id$ and $P=\gamma^{-1}\id$,~\eqref{eqnn:-12} holds with $c_1=c_2=\gamma^{-1}$. Consequently, if $(\lambda_n(\delta-\lambda_n))_{n\in\Nn}\subseteq[\tau,\infty[$ for some $\tau>0$, we have 
 \begin{equation*} 
 \|\tilde{z}_{n}\|^{2}\leq\frac{1}{\tau(n+1)}\|z_{0}-z^{\star}\|^{2},
 \end{equation*}
and $\|\tilde{z}_{n}\|^{2}=o(1/(n+1))$.
\subsection{Solodov and Tseng}
 In \Cref{Algorithm-1}, set $C\equiv0$, $A=N_X$ and $H=\id$, where $N_X$ is the normal cone operator of $X$, and $X$ is a nonempty closed, convex set in $\HH$. Then we recover the scheme proposed by Solodov and Tseng \cite[Algorithm 2.1]{solodov1996modified}. 
 \subsection{Forward-Backward-Forward Splitting}
Consider~\Cref{Algorithm-1} when $M$ is skew-adjoint and set $H=\gamma^{-1} \id$, $S=\id$. We can enforce $\alpha_n=\gamma$  by choosing $\lambda_n=(\gamma\|(\gamma^{-1} \id +M^{*})\tilde{z}_{n}\|/\|\tilde{z}_{n}\|)^{2}$. It remains to show that the sequence $(\lambda_n)_{n\in\Nn}$ satisfies the conditions of \Cref{thm:Suppose-that-}.  
Since $M$ is skew-adjoint, we have $\lambda_n= 1 + (\gamma\|M\tilde{z}_{n}\|/\|\tilde{z}_{n}\|)^{2}$, and if the stepsize satisfies 
$\gamma\in]0,\|M\|^{-1}\sqrt{1-1/(2\beta)}[$, 
then $(\lambda_n)_{n\in\Nn}$ is uniformly bounded between $0$ and $\delta$ (in fact it is larger than $1$) and thus satisfies~\eqref{eq:-7}. Under these assumptions \Cref{Algorithm-1} simplifies to
\begin{subequations}
	\begin{align*}
	\bar{z}_{n} & =(\id+\gamma A)^{-1}(\id-\gamma M-\gamma C)z_{n}\\
	z_{n+1} & =\bar{z}_{n}-\gamma M\left(\bar{z}_n-{z}_n\right).
	\end{align*}
\end{subequations}
This algorithm resembles the FBFS \cite{tseng2000modified}. Indeed, if $C\equiv0$, then the range for the stepsize simplifies to 
$\gamma\in]0,\|M\|^{-1}[$ and yields 
the FBFS when its Lipschitz operator is the skew-adjoint operator $M$.

\subsection{Douglas-Rachford Type with a Forward Term} \label{SUBSEC:DRF}
We now focus our attention on a choice for $P$, $K$ and $S$ that  lead to a new  Douglas-Rachford type splitting with a forward term. In \Cref{sec:3} we consider more general $S$, $P$, $K$ and the algorithm presented here can be derived as a special case in \Cref{DR-ADMM} in which we also discuss a 3-block ADMM algorithm. Consider the problem of finding $x\in \HH$ such that  
\begin{equation}
0\in Dx+Ex+Fx, \label{eq:DR main}
\end{equation}
together with the dual inclusion problem of finding $y\in \HH$ such that there exists $x\in \HH$, 
 \begin{equation} \label{eq:DR primal-dual}
 \begin{cases}
 0\in Dx+Fx+y\\
 0\in E^{-1}y-x.
 \end{cases}
 \end{equation}
where $D:\HH\to2^{\HH}$, $E:\HH\to2^{\HH}$ are maximally monotone and $F:\HH\to\HH$ is 
$\eta$-cocoercive with respect to the canonical norm. Let $\mathcal{K}$ be the Hilbert direct sum  
$\mathcal{K}=\HH\oplus\HH$. The pair $(x^\star,y^\star)\in\mathcal{K}$ is called a primal-dual solution to~\eqref{eq:DR main} if it satisfies~\eqref{eq:DR primal-dual}.
Let $(x^\star,y^\star)\in\mathcal{K}$ be a primal-dual solution, then $x^\star$ solves the primal  problem \eqref{eq:DR main} and $y^\star$ the dual~\eqref{eq:DR primal-dual}. In this section, we assume that there exists $x^{\star}$ such that 
$x^{\star}\in \zer \left(  D + E + F \right)$. This assumption yields that the set of primal-dual solutions is nonempty (see \cite{combettes2012primal,boct2014recent} and the references therein for more discussion). \\ 
Reformulate~\eqref{eq:DR primal-dual} in the form of~\eqref{eq:} by defining
\begin{subequations}
\begin{align}
A:\mathcal{K}\to2^{\mathcal{K}}&:(x,y)\mapsto(Dx,E^{-1}y),\label{eq:Aopt1}\\
M\in\mathscr{B}(\mathcal{K})&:(x,y)\mapsto(y,-x),\label{eq:Mopt1}\\
C:\mathcal{K}\to\mathcal{K}&:(x,y)\mapsto(Fx,0).\label{eq:Copt1}
\end{align}
\end{subequations}
The operators $A$ and $M$ are maximally monotone \cite[Proposition 20.23 and Example 20.30]{bauschke2011convex}. It is easy to verify, by definition of cocoercivity, that $C$ is cocoercive. 
Let $\gamma>0$, $\theta\in [0,2]$,  (the case of $\theta=2$ can only be considered in the absence of the cocoercive term and results in classic DR, see \Cref{prop:-1}). Set
\begin{subequations} \label{eq:-42-all}
\begin{align}
P\in\mathscr{B}(\mathcal{K}):(x,y)&\mapsto\left(\gamma^{-1}x-\tfrac{1}{2}\theta y,-\tfrac{1}{2}\theta x+\gamma y\right), \label{eq:-42-p} \\
K\in\mathscr{B}(\mathcal{K}):(x,y)&\mapsto\left(\tfrac{1}{2}\theta  y,-\tfrac{1}{2}\theta x\right),\label{eq:-42-k} \\
S\in\mathscr{B}(\mathcal{K}):(x,y)&\mapsto\left((3-\theta)\gamma^{-1}x-y,-x+\gamma y\right). \label{eq:-42-S}
\end{align}
\end{subequations}
The operators $P$ and $S$ defined in~\eqref{eq:-42-all} are special cases of~\eqref{eq:-17-1} and~\eqref{eq:S-6} when $L=\id$, $\gamma_1=\gamma_2^{-1}=\gamma$. It follows from \Cref{lem:strongly positive p-1,lem:strongly} that they are strongly positive for $\theta\in[0,2[$. Then $H=P+K$ becomes
\begin{equation}
H\in\mathscr{B}(\mathcal{K}):(x,y)\mapsto(\gamma^{-1}x,-\theta x+\gamma y). \label{eq:-42-H}
\end{equation}
Notice that $H$ has the block triangular structure described in \Cref{lem:separable}. By using this structure as in~\eqref{eq:resolventform} and substituting~\eqref{eq:-42-all},~\eqref{eq:-42-H} in \Cref{Algorithm-1}, after some algebraic manipulations involving Moreau's identity as well as a change of variables $s_{n}\coloneqq x_{n}-\gamma y_{n}$ (see proof of \Cref{prop:-1} for details), we derive the following algorithm:
\begin{algorithm}[H]
    \caption{Douglas-Rachford Type with a Forward Term}
	\label{Algorithm-6}
    \begin{algorithmic} 
            \item\textbf{Inputs:} $x_0\in\HH$, $s_0\in\HH$
            \For{$n=0,\ldots$} 
                \State $\bar{x}_{n} =J_{\gamma D}(s_{n}-\gamma Fx_{n})$
                \State $r_{n}=J_{\gamma E}(\theta \bar{x}_{n}+(2-\theta)x_{n}-s_{n})$
                \State $s_{n+1} =s_{n}+\rho_n(r_{n}-\bar{x}_{n})$
                \State $x_{n+1}=x_n+\rho_n(\bar{x}_{n}-x_n)$
            \EndFor
    \end{algorithmic}
\end{algorithm}
In the special case when $\rho_n=1$, the last line in \Cref{Algorithm-6} becomes obsolete and $\bar{x}_n$ can be replaced with $x_{n+1}$. The next proposition provides the convergence properties for \Cref{Algorithm-6}.
\begin{prop} \label{prop:-1}
Consider the sequences $(x_{n})_{n\in\Nn}$ and $(s_{n})_{n\in\Nn}$ generated
by \Cref{Algorithm-6}. Let $\eta\in]0,+\infty[$ be the
cocoercivity constant of $F$. Suppose that one of the following holds:
\begin{enumerate}[{label=(\textit{\roman*})},ref=\textit{\roman{*}}]
\item \label{prop:-1-part1} $\theta\in[0,2[$, $\gamma\in]0,\eta(4-\theta^2)[$ and the sequence of relaxation parameters $(\rho_{n})_{n\in\Nn}$  is uniformly bounded in the interval
\begin{equation}\label{eq:rho}
(\rho_n)_{n\in\Nn}\subseteq\left]0, \frac{4-\theta^2-\gamma/\eta}{(2-\theta)(2+\sqrt{2-\theta})}\right[. 
\end{equation}
\item \label{prop:-1-part2} $F\equiv0$, $\theta\in[0,2[$,  $\gamma\in]0,\infty[$, and sequence $(\rho_{n})_{n\in\Nn}$ uniformly bounded in the interval
\begin{equation*}
(\rho_n)_{n\in\Nn}\subseteq\left]0,2-\sqrt{2-\theta}\right[.
\end{equation*}
\item \label{prop:-1-part3} $F\equiv0$, $\theta=2$, $\gamma\in]0,\infty[$,  $(\rho_{n})_{n\in\Nn}$ uniformly bounded in the interval $]0,2[$, and $\mathcal{K}$ is finite-dimensional.
\end{enumerate}
Then, there exists a pair of solutions $(x^{\star},y^{\star})\in\mathcal{K}$
to~\eqref{eq:DR primal-dual} such that the sequences 
$(x_{n})_{n\in\Nn}$ and $(s_{n})_{n\in\Nn}$ converge weakly to $x^{\star}$ and $x^{\star}-\gamma y^{\star}$, respectively. 
\end{prop}
\begin{proof}
See \proofRef{proof of prop:-1}.
\end{proof}

Next proposition provides convergence rate results for~\Cref{Algorithm-6} when $\theta=2$, based on~\Cref{thm:Suppose-that--1}. Similarly, for the case when $\theta\in[0,2[$, convergence rates can be deduced based on \Cref{Thm: conv-rates}. However, we omit it in this work. Furthermore, when {metric subregulariy} assumption in \Cref{Thm: conv-rates-ii} holds, linear convergence follows without any additional assumptions.  
\begin{prop}[Convergence rate] \label{porp:-1-conv}
	Let $\mathcal{K}$ be finite-dimensional. Consider the sequences $(x_{n})_{n\in\Nn}$ and $(s_{n})_{n\in\Nn}$ generated by \Cref{Algorithm-6}. Let $F\equiv0$, $\theta=2$, $\gamma\in]0,\infty[$, and $(\rho_{n})_{n\in\Nn}$ be  uniformly bounded in the interval $]0,2[$. Then
\[
\|s_{n+1}-s_n\|^2\leq\tfrac{\gamma}{\tau(n+1)}\|Qz_{0}-Qz^{\star}\|_{R}^{2},
\] 
and $\|s_{n+1}-s_n\|^2=o(1/(n+1))$ for some constant $\tau>0$, 
where $Q$ is the orthogonal projection onto $\ran(P)$, $R=P+\id-Q$, and $z_n= \left(x_n,\gamma^{-1}(x_n-s_n)\right)$. 
\end{prop}
\begin{proof}
See \proofRef{proof of prop:-1-conv}.
\end{proof}
\begin{rem} \label{rem:DR}
Recently, another three-operator splitting algorithm was  proposed in \cite{davis2015three} which can also be seen as a generalization of Douglas-Rachford method to accommodates a third cocoercive operator. In the aforementioned paper, the forward step takes place after the first backward update, while in \Cref{Algorithm-6} it precedes the backward update. Whether this is better in practice or not is yet to be seen.

Furthermore, the parameter range prescribed in \cite[Theorem 3.1]{davis2015three} is simply $\gamma\in]0,2\eta[$ and $(\rho_n)_{n\in\Nn}\subseteq]0,2-\frac{\gamma}{2\eta}[$, while for \Cref{Algorithm-6} it consists of $\gamma\in]0,\eta(4-\theta^2)[$ with relaxation parameter uniformly bounded in the interval~\eqref{eq:rho}. For $\theta\in[0,\sqrt{2}[$, \Cref{Algorithm-6} can have larger step size but it is important to notice that this might not necessarily be advantageous in practice because the upper bound for the relaxation parameter in \eqref{eq:rho} decreases as we reduce $\theta$. For example if we fix $\rho_n=1$, conditions of \Cref{prop:-1} become $\theta\in]1,2[$ and $\gamma/\eta\in]0,(2-\theta)(\theta-\sqrt{2-\theta})[$. This step size is always smaller that the one of \cite{davis2015three}. However, if the relaxation parameter $\rho_n$ is selected to be small enough then  $\gamma$ can take values larger than the one allowed in \cite{davis2015three}.
\end{rem}
 
  \begin{rem}
  	In \Cref{Algorithm-6} the case $\theta=2$, $\rho_n=1$ with $F\equiv0$ (see \Cref{prop:-1}\eqref{prop:-1-part3}) yields the classical DRS \cite{lions1979splitting}. This choice of $P$ is precisely the one considered in \cite[Section 3.1.1]{condat2013primal}. 
  \end{rem}

\section{Structured Nonsmooth Convex Optimization} \label{sec:Saddle-point-problem}
One of the characteristics of Asymmetric Forward-Backward-Adjoint splitting (AFBA) introduced in \Cref{sec:3}, is availability of the parameters $P,K$ and $S$ which are independent of each other. In the general form, $P$ is a strongly positive operator to be chosen but it directly effects the convergence through the cocoercivity constant in \Cref{assumption-1}\eqref{enu:Let--and}. $S$ and $K$ are arbitrary strongly positive and skew-adjoint operators, respectively. This introduces a lot of flexibility which proves essential in the development of this section. We will see that by properly choosing these parameters we can recover and generalize several well known schemes proposed in the recent years. Specifically, we will recover the algorithms proposed by  V\~{u} and Condat \cite{vu2013splitting,condat2013primal}, Brice{\~n}o-Arias and Combettes \cite{briceno2011monotone+} and the one of Drori and Sabach  and Teboulle \cite{drori2015simple}. These algorithms belong to the class of so called primal-dual algorithms and owe their popularity to their simplicity and special structure. They have been used to solve a wide range of problems arising in image processing, machine learning and control, see for example \cite{condat2014generic,chambolle2011first,boct2014recent}. Recently, randomized versions have also been proposed for distributed optimization (refer to \cite{pesquet2014class,bianchi2015coordinate,combettes2014stochastic}). As the last contribution of the paper we will present a generalization of the classic ADMM to three blocks. \Cref{fig1} summarizes the connections between these algorithms in the form of a diagram. 
\begin{figure}
\raggedleft
\begin{minipage}{0.9\columnwidth}
\def\svgwidth{15cm}
\resizebox{\columnwidth}{!}{\trimbox{1.8cm 8.8cm 0 3.6cm}{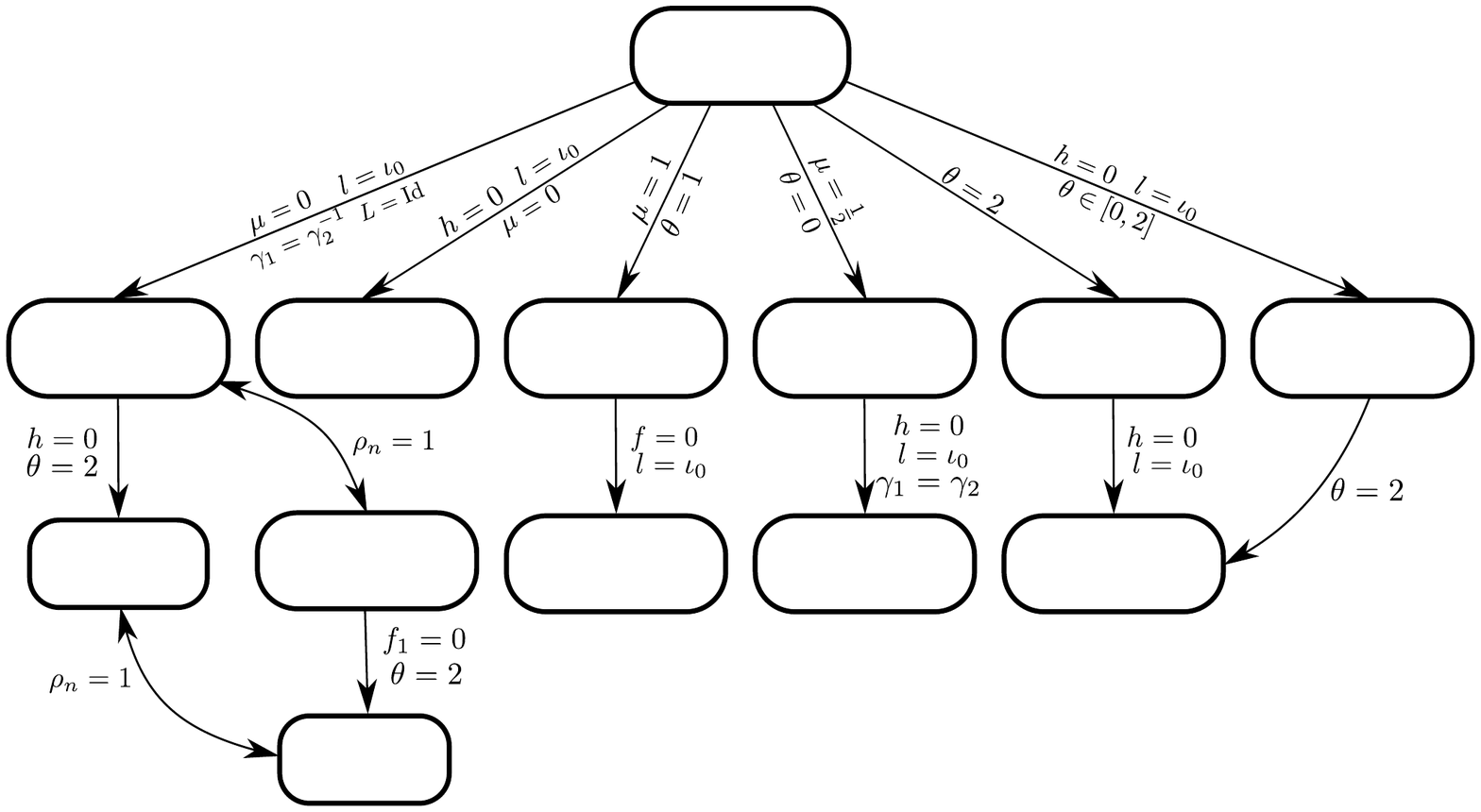 }}
\end{minipage}
\caption{\Cref{Algorithm-2} and its special cases}\label{fig1}
\end{figure}
 We start with the following convex optimization problem
\begin{align}
\underset{x\in\HH}{\minimize}\ f(x)+h(x)+(g\oblong l)(Lx), \label{eq:-21-1}
\end{align}
and its dual 
\begin{align}
\underset{y\in\mathcal{G}}{\minimize}\ (f+h)^{*}(-L^*y)+(g\oblong l)^*(y), \label{eq:-21-1:Dual}
\end{align}
 where $\mathcal{H}$ and $\mathcal{G}$ are real Hilbert spaces, $f\in\Gamma_{0}(\HH)$, $g\in\Gamma_{0}(\mathcal{G})$,
$L\in\mathscr{B}(\HH,\mathcal{G})$ and $h:\HH\rightarrow\mathbb{R}$
is convex, differentiable on $\HH$ with $\beta_{h}$-Lipschitz gradient
for some $\beta_{h}\in]0,+\infty[$. Furthermore, $l\in\Gamma_{0}(\mathcal{G})$
is $\beta_{l}^{-1}$-strongly convex for some $\beta_{l}\in]0,+\infty[$ or equivalently $\nabla l^{*}$ is $\beta_{l}$-Lipschitz. 
 The infimal convolution $g\oblong l$  can be seen as a regularization of $g$ by $l$ and when $l=\iota_{\{0\}}$, $(g\oblong l)$ simply becomes $g$ and $\nabla l^{*}=0$. In our analysis we will always consider the special cases $l=\iota_{\{0\}}$ and $h\equiv0$ separately because they result in less conservative conditions.    

Throughout this section we assume that there exists $x^{\star}$ such that 
\begin{equation} \label{eqn:existance}
x^{\star}\in \zer \left(  \partial f + \nabla h + L^*(\partial g \oblong \partial l)L    \right).
\end{equation}
The interested reader is referred to \cite[Proposition 4.3]{combettes2012primal} for conditions on existence of such $x^{\star}$. We consider the saddle point problem corresponding to \eqref{eq:-21-1}. This allows us to exploit~\Cref{Algorithm-1}, and develop a unifying algorithm. We do all this in the context of optimization, with the understanding that it is straightforward to adapt the same analysis for the corresponding monotone inclusion problem. The saddle point problem is 
\begin{equation*}
\underset{x\in\HH}{\minimize}\ \underset{y\in\mathcal{G}}{\maximize}\ f(x)+h(x)+\langle Lx,y\rangle-(g\oblong l)^{*}(y). \label{eq:saddle-point}
\end{equation*}
The optimality conditions are
\begin{equation} \label{eq:primal-dual}
\begin{cases}
0\in\partial f(x)+\nabla h(x)+L^{*}y,\\
0\in\partial g^{*}(y)+\nabla l{}^{*}(y)-Lx.
\end{cases} 
\end{equation}
It follows from~\eqref{eqn:existance} that the set of solutions to \eqref{eq:primal-dual} is nonempty. We say that $(x^{\star},y^{\star})$ is a primal-dual solution if it satisfies~\eqref{eq:primal-dual}. 
Furthermore, if $(x^{\star},y^{\star})$ is a solution pair to \eqref{eq:primal-dual}, then $x^{\star}$ is a solution for the primal problem \eqref{eq:-21-1} and $y^{\star}$ is a solution to the dual problem~\eqref{eq:-21-1:Dual}. For further discussion on duality see \cite{combettes2012primal,davis2015convergence,condat2013primal} and the references therein. 
Let $\mathcal{K}$ be the Hilbert direct sum  
$\mathcal{K}=\HH\oplus\mathcal{G}$ and
define the operators 
\begin{subequations}
\begin{align}
A:\mathcal{K}\to2^{\mathcal{K}}&:(x,y)\mapsto(\partial f(x),\partial g^{*}(y)),\label{eq:Aopt}\\
M\in\mathscr{B}(\mathcal{K})&:(x,y)\mapsto(L^{*}y,-Lx),\label{eq:Mopt}\\
C:\mathcal{K}\to\mathcal{K}&:(x,y)\mapsto(\nabla h(x),\nabla l{}^{*}(y)).\label{eq:Copt}
\end{align}
\end{subequations}
The operator $A$ is maximally monotone \cite[Theorem 21.2 and Proposition 20.23]{bauschke2011convex}, and the operator $C$ is cocoercive (see \Cref{lem: cocoercivity lemma}). The monotone inclusion problem \eqref{eq:primal-dual} can be written in the form of \eqref{eq:}: 
\[
0\in Az+Mz+Cz,
\]
where $z=(x,y)$. Let $\theta\in [0,+\infty[$ and set 
\begin{align}
P\in\mathscr{B}(\mathcal{K}):(x,y)&\mapsto\left(\gamma_{1}^{-1}x-\tfrac{1}{2}\theta L^{*}y,-\tfrac{1}{2}\theta Lx+\gamma_{2}^{-1}y\right),\label{eq:-17-1}\\
K\in\mathscr{B}(\mathcal{K}):(x,y)&\mapsto\left(\tfrac{1}{2}\theta L^{*}y,-\tfrac{1}{2}\theta Lx\right).\nonumber
\end{align}
Then $H=P+K$ yields
\begin{equation}
H\in\mathscr{B}(\mathcal{K}):(x,y)\mapsto\left(\gamma_{1}^{-1}x,-\theta Lx+\gamma_{2}^{-1}y\right).\label{eq:-42-42-42}
\end{equation}
\begin{rem}
We make the above choices for clarity of exposition. It is straightforward to adopt the same analysis when $\gamma_{1}$ and $\gamma_{2}$ are replaced by general strongly positive operators $\Gamma_1$ and $\Gamma_2$. 
\end{rem}
Operator $H$ defined in \eqref{eq:-42-42-42} has the block triangular structure described in~\Cref{lem:separable}. Therefore, in view of~\eqref{eq:resolventform}, $\bar{z}_{n}=(H+A)^{-1}(H-M-C)z_{n}$ becomes
\begin{subequations}\label{eq:bar_prox}
\begin{align}
\bar{x}_{n}&=\prox_{\gamma_1 f}(x_{n}-\gamma_{1}L^{*}y_{n}-\gamma_{1}\nabla h(x_{n})),\\
\bar{y}_{n}&=\prox_{\gamma_2g^*}(y_{n}+\gamma_{2}L((1-\theta) x_{n}+\theta\bar{x}_{n})-\gamma_{2}\nabla l^{*}(y_{n})).
\end{align}
\end{subequations}
	The following two lemmas will play an important role in development of this section. \Cref{lem:strongly positive p-1} provides a tight estimate for the strong positivity parameter of $P$, while in \Cref{lem: cocoercivity lemma} we develop estimates for the cocoercivity parameter of $C$ given by~\eqref{eq:Copt}. 
\begin{lem}
\label{lem:strongly positive p-1}Consider $P$ defined by \eqref{eq:-17-1}. 
Let $\gamma_{1},\gamma_{2}>0$ and 
\begin{equation}
\gamma_{1}^{-1}-\frac{\gamma_{2}}{4}\theta^{2}\|L\|^{2}>0.\label{eq:-41}
\end{equation}
Then $P\in\mathcal{S}_\tau(\KK)$ where
\begin{equation}
\tau=\tfrac{1}{2}\gamma_{1}^{-1}+\tfrac{1}{2}\gamma_{2}^{-1}-\tfrac{1}{2}\sqrt{\theta^{2}\|L\|^{2}+(\gamma_{1}^{-1}-\gamma_{2}^{-1})^{2}}.\label{eq:-42-42}
\end{equation}
 \end{lem}
\begin{proof}
Let $z=(x,y)$. First consider the case when $\|L\|>0$ and $\theta\in]0,\infty[$. We have
\begin{align}
	\langle z,Pz\rangle 
	& =\langle\gamma_{1}^{-1}x-\tfrac{1}{2}\theta L^{*}y,x\rangle+\langle-\tfrac{1}{2}\theta Lx+\gamma_{2}^{-1}y,y\rangle\nonumber \\
	& =\gamma_{1}^{-1}\|x\|^{2}+\gamma_{2}^{-1}\|y\|^{2}-\theta\langle Lx,y\rangle\nonumber \\
	& \geq(\gamma_{1}^{-1}-\tfrac{\theta\|L\|^{2}}{2\epsilon})\|x\|^{2}+(\gamma_{2}^{-1}-\tfrac{\theta\epsilon}{2})\|y\|^{2},\label{eq:-13}
\end{align}
where we used the Fenchel-Young inequality for $\frac{\epsilon}{2}\|\cdot\|^2$. Select 
$\epsilon=-\tfrac{1}{\theta}(\gamma_{1}^{-1}-\gamma_{2}^{-1})+\sqrt{\|L\|^{2}+\tfrac{1}{\theta^{2}}(\gamma_{1}^{-1}-\gamma_{2}^{-1})^{2}}$,
to maximize the strong positivity parameter. It is
easy to verify that $\epsilon>0$. Substitute $\epsilon$ in \eqref{eq:-13} to obtain
$\langle z,Pz\rangle\geq\tau\|z\|^{2}$, 
where $\tau$ is given by~\eqref{eq:-42-42}
and is positive as long as \eqref{eq:-41} is satisfied. The case when either $\|L\|=0$ or $\theta=0$ results in $\tau$-strongly positive $P$ with $\tau=\min \{\gamma_1^{-1},\gamma_2^{-1}\}$, which is also captured by~\eqref{eq:-42-42}.  
\end{proof}
\begin{lem} \label{lem: cocoercivity lemma}
Let $\beta_{h}\in]0,+\infty[$ and $\beta_{l}\in]0,+\infty[$ be the
Lipschitz constants of $\nabla h$ and $\nabla l^{*}$, respectively. Let $P$ be defined by \eqref{eq:-17-1} and assume that $\gamma_1,\gamma_2>0$ are such that~\eqref{eq:-41} is satisfied.
 Then $C$ given by~\eqref{eq:Copt} is $\beta$-cocoercive with respect to the $P$ norm, where 
\begin{equation}
\beta=\tau \min \{\beta_{l}^{-1},\beta_{h}^{-1}\} \label{eq:-42-43},
\end{equation}
with $\tau$ defined in \eqref{eq:-42-42}.  
If in addition $l=\iota_{\{0\}}$ then 
\begin{equation}
\beta=\beta_{h}^{-1}\left(\gamma_{1}^{-1}-\frac{\gamma_{2}}{4}\theta^{2}\|L\|^{2}\right).\label{eq:-30}
\end{equation}
\end{lem}
\begin{proof}
See \proofRef{proof of lem: cocoercivity lemma}.
\end{proof}
\begin{rem}
It is easy to derive a more conservative strong positivity parameter in \Cref{lem:strongly positive p-1} similar to \cite[Equation (3.20)]{vu2013splitting}:
\begin{equation}
\tau = \min \{\gamma^{-1}_1,\gamma^{-1}_2\}\left(1-\frac{\theta}{2}\sqrt{\gamma_1 \gamma_2 \|L\|^2}\right). \label{42:-41}
\end{equation} 
It must be clear that $\tau$ in \eqref{eq:-42-43} can be replaced with \eqref{42:-41}. However, \eqref{42:-41} does not result in simplification of our convergence analysis and 
the strong positivity parameter \eqref{eq:-42-42} is larger (less conservative) than \eqref{42:-41}, as long as $P$ is a strongly positive operator. Notice that according to \eqref{eq:-42-43} larger $\tau$ results in larger cocoercivity parameter for $C$ and hence is less conservative. A more general version of \Cref{lem: cocoercivity lemma} for several composite functions can be easily derived in a similar way but we will not consider it in this article. 
\end{rem}
 \Cref{Algorithm-1} gives us an extra degree of freedom in choosing $S$. Our aim here is to select $S$ so as to derive an easy to implement scheme without sacrificing  flexibility and generality of the algorithm. To this end, let us define $S_1\in\mathscr{B}(\mathcal{K})$, $S_{2}\in\mathscr{B}(\mathcal{K})$ as follows:
 \begin{subequations}   \label{eq:S}
 	{\mathtight \begin{align} 
 		S_{1}&:(x,y)\mapsto(\gamma_{1}^{-1}x+(1-\theta)L^{*}y,(1-\theta)Lx+\gamma_{2}^{-1}y+\gamma_{1}(1-\theta)(2-\theta)LL^{*}y), \label{eq:S-5}\\
 		S_{2}&:(x,y)\mapsto(\gamma_{1}^{-1}x+\gamma_{2}(2-\theta)L^{*}Lx-L^{*}y,-Lx+\gamma_{2}^{-1}y).\label{eq:S-6} 
 	\end{align} }
 \end{subequations}
  Then, let $\mu\in[0,1]$ and define   
  \begin{equation}\label{eqn:S}
  S=\left(\mu S_{1}^{-1}+(1-\mu)S_{2}^{-1}\right)^{-1},
  \end{equation}
  so that $S^{-1}$ is a convex combination of $S_1^{-1}$ and $S_2^{-1}$. For operator $D$ defined in~\eqref{eq:D} we have 
 \begin{equation}\label{eqn:D}
 D=\mu D_{1}+(1-\mu)D_{2},
 \end{equation}
  for $\mu\in[0,1]$, where $D_1\in\mathscr{B}(\mathcal{K})$ and  $D_{2}\in\mathscr{B}(\mathcal{K})$ follow from~\eqref{eq:D} by substituting $S_1$ and $S_2$, respectively:
 \begin{subequations}   \label{eqnn:D}
 	{\mathtight
\begin{align} 
D_{1}&:(x,y)\mapsto(\gamma_{1}^{-1}x-L^{*}y,-Lx+\gamma_{2}^{-1}y+\gamma_{1}(2-\theta)LL^{*}y), \label{eq:-5}\\
D_{2}&:(x,y)\mapsto(\gamma_{1}^{-1}x+\gamma_{2}(1-\theta)(2-\theta)L^{*}Lx+(1-\theta)L^{*}y,(1-\theta)Lx+\gamma_{2}^{-1}y).\label{eq:-6} 
 \end{align} }
 \end{subequations} 
This choice is simple enough, yet it allows us to unify and generalize several well known methods. The following lemma establishes conditions for strong positivity of $S,D$ and will be used throughout this chapter. 
\begin{lem}
\label{lem:strongly} Let  $\gamma_{1},\gamma_{2}>0$  and assume that \eqref{eq:-41} holds. Then $S$ and $D$ defined in
\eqref{eqn:S} and \eqref{eqn:D} are
strongly positive.\end{lem}
\begin{proof}
See \proofRef{proof of lem:strongly}.
\end{proof}
\Cref{lem:strongly} shows that this choice of $S$  poses no additional constraint on the parameters because it is strongly positive under the same condition required for strong positivity of $P$. By using~\eqref{eq:bar_prox} and $S$ defined in~\eqref{eqn:S} (or equivalently $D$ defined in~\eqref{eqn:D}), we derive the following algorithm from \Cref{Algorithm-1}.
\begin{algorithm}[H] 
    \caption{}   \label{Algorithm-2}
    \begin{algorithmic} 
            \item\textbf{Inputs:} $x_0\in\HH$, $y_0\in\mathcal{G}$
            \For{$n=0,\ldots$} 
                \State $\bar{x}_{n} =\prox_{\gamma_{1}f}(x_{n}-\gamma_{1}L^{*}y_{n}-\gamma_{1}\nabla h(x_{n}))$
                \State $ \displaystyle \bar{y}_{n}=\prox_{\gamma_{2}g^{*}}(y_{n}+\gamma_{2}L((1-\theta)x_{n}+\theta\bar{x}_{n})-\gamma_{2}\nabla l^{*}(y_{n}))$
                \State $\displaystyle \tilde{x}_{n} =\bar{x}_{n}-x_{n},\quad\tilde{y}_{n}=\bar{y}_{n}-y_{n}$
                \State  Compute  $\displaystyle \alpha_{n}$ according to \eqref{eq:alpha}
                \State $x_{n+1} =x_{n}+\alpha_{n}(\tilde{x}_{n}-\mu\gamma_{1}(2-\theta)L^{*}\tilde{y}_{n})$
                \State $y_{n+1} =y_{n}+\alpha_{n}(\gamma_{2}(1-\mu)(2-\theta)L\tilde{x}_{n}+\tilde{y}_{n})$
            \EndFor
    \end{algorithmic} 
    \end{algorithm} 
\begin{align}
\alpha_{n} & =\lambda_{n}\frac{\gamma_{1}^{-1}\|\tilde{x}_{n}\|^{2}+\gamma_{2}^{-1}\|\tilde{y}_{n}\|^{2}-\theta\langle\tilde{x}_{n},L^{*}\tilde{y}_{n}\rangle}{V(\tilde{x}_{n},\tilde{y}_{n})}, \label{eq:alpha}
\end{align}
where 
\begin{align}
V(\tilde{x}_{n},\tilde{y}_{n}) & =\gamma_{1}^{-1}\|\tilde{x}_{n}\|^{2}+\gamma_{2}^{-1}\|\tilde{y}_{n}\|^{2}+(1-\mu)\gamma_{2}(1-\theta)(2-\theta)\|L\tilde{x}_{n}\|^{2}\nonumber\\
 & +\mu\gamma_{1}(2-\theta)\|L^{*}\tilde{y}_{n}\|^{2}+2((1-\mu)(1-\theta)-\mu)\langle\tilde{x}_{n},L^{*}\tilde{y}_{n}\rangle.\label{eqn:-3}
\end{align}
The convergence properties of \Cref{Algorithm-2} are stated in the following theorem. When $l$ is the indicator of $\{0\}$ or $h\equiv0$, it is possible to derive less conservative conditions. \Cref{cor:coco}  distinguishes three cases: the general case; the case with $l=\iota_{\{0\}}$; and the case when $h\equiv0$, $l=\iota_{\{0\}}$. The analysis for $h\equiv0$ and general $l$ is similar.  
\begin{thm} \label{cor:coco}
Consider the sequences $(x_{n})_{n\in\Nn}$ and $(y_{n})_{n\in\Nn}$ generated
by \Cref{Algorithm-2}. Let $\beta_{h}\in]0,+\infty[$ and $\beta_{l}\in]0,+\infty[$ be the
Lipschitz constants of $\nabla h$ and $\nabla l^{*}$, respectively. Let $\theta\in[0,\infty[$ and suppose that one of the following holds:
\begin{enumerate}[{label=(\textit{\roman*})},ref=\textit{\roman{*}}]
\item \label{coco:part1} $(\lambda_{n})_{n\in\Nn}$ is such that \eqref{eq:-7} holds with  
\begin{equation}
\delta=2-\frac{1}{2\tau}\max\{\beta_{l},\beta_{h}\}, \label{eq:dd}
\end{equation}
where $\tau$ is defined in \eqref{eq:-42-42} and 
\begin{equation} \label{eq:panos3}
4\tau\min\{\beta_{h}^{-1},\beta_{l}^{-1}\}>1.
\end{equation}
\item \label{coco:part2} $l=\iota_{\{0\}}$, $(\lambda_{n})_{n\in\Nn}$ is such that~\eqref{eq:-7} holds with 
\begin{equation*} 
	\delta=2-\tfrac{\beta_{h}}{2}\left(\gamma_{1}^{-1}-\tfrac{\gamma_{2}}{4}\theta^{2}\|L\|^{2}\right)^{-1} \quad \textrm{and} \quad \gamma_{1}^{-1}-\tfrac{\gamma_{2}}{4}\theta^{2}\|L\|^{2}>\tfrac{\beta_{h}}{4}.
\end{equation*}
\item \label{coco:part3} $l=\iota_{\{0\}}$, $h\equiv0$ and $(\lambda_{n})_{n\in\Nn}$ is such that~\eqref{eq:-7} holds with $\delta=2$ and 
\begin{equation} \label{eq:-122-P}
\gamma_{1}^{-1}-\frac{\gamma_{2}}{4}\theta^{2}\|L\|^{2}>0.
\end{equation} 
\item \label{coco:part4} $l=\iota_{\{0\}}$, $h\equiv0$, $\theta=2$ and $\mathcal{K}$ is finite-dimensional. $(\lambda_{n})_{n\in\Nn}$ is uniformly bounded in the interval $]0,2[$ and 
$\gamma_{1}^{-1}-\gamma_{2}\|L\|^{2}\geq0$.
\end{enumerate}
Then there exists a pair of solutions $(x^{\star},y^{\star})\in\mathcal{K}$
to
~\eqref{eq:primal-dual} such that the sequences
$(x_{n})_{n\in\Nn}$ and $(y_{n})_{n\in\Nn}$ converge weakly to $x^{\star}$ and $y^{\star}$, respectively. 
\end{thm}
\begin{proof}
\eqref{coco:part1}: The proof relies on \Cref{thm:Suppose-that-}. From~\eqref{eq:panos3} we have that~\eqref{eq:-41} holds.  Therefore, strong positivity of $P$ and $S$ follow from~\Cref{lem:strongly positive p-1,lem:strongly}. Use \Cref{lem: cocoercivity lemma} and substitute~\eqref{eq:-42-43}  in~\eqref{eq:-7} to derive~\eqref{eq:dd} and~\eqref{eq:panos3}. Therefore, weak convergence follows by invoking \Cref{thm:Suppose-that-}. 
\eqref{coco:part2}: Mimic the proof of \eqref{coco:part1} but replace the cocoercivity parameter, $\beta$, with~\eqref{eq:-30} instead of~\eqref{eq:-42-43}. \eqref{coco:part3}: The  proof is similar to previous parts, except that we use \Cref{thm:Suppose-that-} with $C\equiv0$.  \eqref{coco:part4}: In this case $P\in \mathcal{S}(\mathcal{K})$, $P\succeq0$.
In view of \eqref{eq:Aopt},~\eqref{eq:-42-42-42} and \Cref{lem:separable} $(H+A)^{-1}P$  is continuous,  and thus~\Cref{thm:Suppose-that--1} completes the proof. 
\end{proof}
\begin{rem}
Condition~\eqref{eq:panos3} of \Cref{cor:coco} can be simplified for some practical special cases. For example, if  ${\mathtight\gamma_1\!=\!\gamma_2\!=\!\gamma}$, \eqref{eq:panos3} becomes ${\mathtight \gamma\!<\!{4}\left(\!2\theta \|L\|\!+\!\max\{\beta_h,\!\beta_l\}\!\right)^{-1}
}$.
\end{rem}
\begin{rem}
\label{rem:()We-note-that} \Cref{Algorithm-2} requires 4 matrix vector products in the general case. However, we consider
two special cases $\mu=1$ and $\mu=0$ in which only 2 matrix vector products are needed. Consider the case where $\mu=0$. Initially calculate
$Lx_{0}$ and $L\bar{x}_{0}$. This yields $L\tilde{x}_{0}$, 
$\alpha_{0}$ and consequently $Lx_{1}=Lx_{0}+\alpha_{0}L\tilde{x}_{0}$.
Continue in the same fashion. Therefore at each iteration we only
need to calculate $L\bar{x}_{n}$. A similar trick can be applied
when $\mu=1$.
\end{rem}
\begin{rem}
 \Cref{Algorithm-2}, is not symmetric with respect to the primal and dual variables.
Switching the role of the primal and dual variables we obtain 
\begin{align*}
\bar{y}_{n} & =\prox_{\gamma_{2}g^{*}}(y_{n}+\gamma_{2}Lx_{n}-\gamma_2 \nabla l^{*}(y_n))\\
\bar{x}_{n} & =\prox_{\gamma_{1}f}(x_{n}-\gamma_{1}L^{*}((1-\theta)y_{n}+\theta\bar{y}_{n})-\gamma_{1}\nabla h(x_{n}))\\
\tilde{x}_{n} & =\bar{x}_{n}-x_{n},\tilde{y}_{n}=\bar{y}_{n}-y_{n}\\
y_{n+1} & =y_{n}+\alpha_{n}(\mu(2-\theta)\gamma_{2}L\tilde{x}_{n}+\tilde{y}_{n})\\
x_{n+1} & =x_{n}+\alpha_{n}(\tilde{x}_n-(1-\mu)(2-\theta)\gamma_{1}L^{*}\tilde{y}_n),
\end{align*}
with
\begin{align*}
\alpha_{n} & =\lambda_{n}\frac{\gamma_{1}^{-1}\|\tilde{x}_n\|^{2}+\gamma_{2}^{-1}\|\tilde{y}_n\|^{2}+\theta\langle\tilde{x}_n,L^{*}\tilde{y}_n\rangle}{V(\tilde{x}_{n},\tilde{y}_{n})},
\end{align*}
where
\begin{align*}
V(\tilde{x}_{n},\tilde{y}_{n}) & =\gamma_{1}^{-1}\|\tilde{x}_n\|^{2}+\gamma_{2}^{-1}\|\tilde{y}_n\|^{2}+\mu\gamma_{2}(2-\theta)\|L\tilde{x}_n\|^{2}\\
 & +(1-\mu)\gamma_{1}(1-\theta)(2-\theta)\|L^{*}\tilde{y}_n\|^{2}\\
 & -2((1-\mu)(1-\theta)-\mu)\langle\tilde{x}_n,L^{*}\tilde{y}_n\rangle,
\end{align*}
for $\mu\in[0,1]$ and $\theta\in[0,+\infty[$. 
\end{rem}

\subsection{ Special Case \texorpdfstring{$\theta=2$}{}} \label{sub:theta=00003D2}
One notable special case of \Cref{Algorithm-2} is the algorithm proposed by V\~{u} and Condat in \cite{vu2013splitting,condat2013primal}.
Substitute $\theta=2$ in~\Cref{Algorithm-2} to obtain
\begin{align*}
\bar{x}_{n} & =\prox_{\gamma_{1}f}(x_{n}-\gamma_{1}L^{*}y_{n}-\gamma_{1}\nabla h(x_{n}))\\
\bar{y}_{n} & =\prox_{\gamma_{2}g^{*}}(y_{n}+\gamma_{2}L(2\bar{x}_{n}-x_{n})-\gamma_{2}\nabla l^{*}(y_{n}))\\
\tilde{x}_{n} & =\bar{x}_{n}-x_{n},\tilde{y}_{n}=\bar{y}_{n}-y_{n}\\
x_{n+1} & =x_{n}+\lambda_{n}\tilde{x}_{n}\\
y_{n+1} & =y_{n}+\lambda_{n}\tilde{y}_{n}.
\end{align*}
The convergence results are summarized in the following proposition and they are direct consequences of \Cref{cor:coco,Thm: conv-rates,thm:Suppose-that--1} and for this reason the proof is omitted. 
\begin{prop}
\label{prop:condat}
Consider the sequences $(x_{n})_{n\in\Nn}$ and $(y_{n})_{n\in\Nn}$ generated
by \Cref{Algorithm-4}. Let $\beta_{h}\in]0,+\infty[$ and $\beta_{l}\in]0,+\infty[$ be the
Lipschitz constants of $\nabla h$ and $\nabla l^{*}$, respectively. The following results hold:
\begin{enumerate}
\item \label{prop:condat:1} (Convergence) Suppose one of the following holds: 
\begin{enumerate}[{label=(\textit{\roman*})},ref=\textit{1-\roman{*}}]
\item \label{prop:condat:part1} $(\lambda_{n})_{n\in\Nn}$ is such that \eqref{eq:-7} holds where $\delta$ is defined in~\eqref{eq:dd} with $\theta=2$.  
\item \label{prop:condat:part2} $l=\iota_{\{0\}}$, $(\lambda_{n})_{n\in\Nn}$ is such that~\eqref{eq:-7} holds with $\delta\!=\!2-\tfrac{\beta_{h}}{2}\!\left(\gamma_{1}^{-1}\!-\!\gamma_{2}\|L\|^{2}\right)\!^{-1}$, 
 and
\begin{equation*} 
\gamma_{1}^{-1}-\gamma_{2}\|L\|^{2}>\frac{\beta_{h}}{4}.
\end{equation*}
\item \label{prop:condat:part3} $l=\iota_{\{0\}}$, $h\equiv0$ and $(\lambda_{n})_{n\in\Nn}$ is such that~\eqref{eq:-7} holds with $\delta=2$ and 
\[
\gamma_{1}^{-1}-\gamma_{2}\|L\|^{2}>0,
\]  
\item \label{prop:condat:part4} $l=\iota_{\{0\}}$, $h\equiv0$ and $\mathcal{K}$ is finite-dimensional. $(\lambda_{n})_{n\in\Nn}$ is uniformly bounded in the interval $]0,2[$  and $\gamma_{1}^{-1}-\gamma_{2}\|L\|^{2}\geq0$.
\end{enumerate}
Then there exists a pair of solutions
$(x^{\star},y^{\star})\in\mathcal{K}$ to 
\eqref{eq:primal-dual} such that the sequences
$(x_{n})_{n\in\Nn}$ and $(y_{n})_{n\in\Nn}$ converge weakly to $x^{\star}$ and $y^{\star}$, respectively.
\item \label{prop:condat:2} (Convergence rate) 
\begin{enumerate}[{label=(\textit{\roman*})},ref=\textit{2-\roman{*}}]
\item Let $(\lambda_n(\delta-\lambda_n))_{n\in\Nn}\subseteq[\underline{\tau},\infty[$ for some $\underline{\tau}>0$. For \eqref{prop:condat:part1}-\eqref{prop:condat:part3} the following convergence rates hold:
\begin{equation*} 
\|\tilde{z}_{n}\|_{P}^{2}\leq\frac{1}{\underline{\tau}(n+1)}\|z_{0}-z^{\star}\|_{P}^{2}\quad\textrm{and}\quad \|\tilde{z}_{n}\|_{P}^{2}=o(1/(n+1)).
\end{equation*}
\item For \eqref{prop:condat:part4} the convergence rates in \eqref{positiv-P-rate} holds, with $R,Q$ defined therein.  
\end{enumerate}
\end{enumerate}
\end{prop}
 \begin{rem}
The conditions in~\Cref{prop:condat}\eqref{prop:condat:part1} are less restrictive than the conditions of \cite[Theorem 3.1]{vu2013splitting}. In the case when $l=\iota_{\{0\}}$, a less conservative condition was proposed in~\cite[theorem 3.1]{condat2013primal}, which is also more conservative than \Cref{prop:condat}\eqref{prop:condat:part2}. The difference is that in the latter $\gamma_{1}^{-1}-{\gamma_{2}}\|L\|^{2}$ needs to be larger than $\beta_h/4$ rather than larger or equal to $\beta_h/2$. In the special case when $\mathcal{K}$ is finite-dimensional, \Cref{prop:condat}\eqref{prop:condat:part4} coincides with ~\cite[theorem 3.3]{condat2013primal}. \Cref{prop:condat}\eqref{prop:condat:2} Further provides big-$O$ and little-$o$ convergence rates by only requiring $(\lambda_n(2-\lambda_n))_{n\in\Nn}\subseteq[\underline{\tau},\infty[$ for some $\underline{\tau}>0$. 
 \end{rem}

\subsection{Special Case \texorpdfstring{$\theta=0$, $\mu=\frac{1}{2}$}{ }}
Another special case of \Cref{Algorithm-2} generalizes the algorithm proposed by Brice{\~n}o-Arias and Combettes in \cite{briceno2011monotone+}. Specifically, setting $\theta=0$ in \Cref{Algorithm-2}, the Gauss-Seidel-type proximal updates are lost but positive definiteness of $P$ boils down to $\gamma_{i}>0$, for $i=1,2$. In addition, we set $\mu=\frac{1}{2}$ and select $\lambda_n$ such that the algorithm converges with constant stepsize $\alpha_n=1$ (see the proof of \Cref{lem:mu=00003D1/2}). This leads to~\Cref{Algorithm-4}: 
\begin{algorithm}[H]
    \caption{}
	\label{Algorithm-4}
    \begin{algorithmic} 
            \item\textbf{Inputs:} $x_0\in\HH$, $y_0\in\mathcal{G}$
            \For{$n=0,\ldots$} 
                \State $\bar{x}_{n}  =\prox_{\gamma_{1}f}(x_{n}-\gamma_{1}L^{*}y_{n}-\gamma_{1}\nabla h(x_{n}))$
                \State $\bar{y}_{n}  =\prox_{\gamma_{2}g^{*}}(y_{n}+\gamma_{2}Lx_{n}-\gamma_{2}\nabla l^{*}(y_{n}))$
                \State $\tilde{x}_{n}  =\bar{x}_{n}-x_{n},\tilde{y}_{n}=\bar{y}_{n}-y_{n}$
                \State $x_{n+1}  =\bar{x}_{n}-\gamma_{1}L^{*}\tilde{y}_{n}$
                \State $y_{n+1} =\bar{y}_{n}+\gamma_{2}L\tilde{x}_{n}$
            \EndFor
    \end{algorithmic}
\end{algorithm}
The convergence results for~\Cref{Algorithm-4} are stated in the following proposition. 
\begin{prop}
\label{lem:mu=00003D1/2}
Consider the sequences $(x_{n})_{n\in\Nn}$ and $(y_{n})_{n\in\Nn}$ generated
by \Cref{Algorithm-4}. Let $\beta_{h}\in]0,+\infty[$ and $\beta_{l}\in]0,+\infty[$ be the
Lipschitz constants of $\nabla h$ and $\nabla l^{*}$, respectively. Suppose that one of the following holds: 
\begin{enumerate}[{label=(\textit{\roman*})},ref=\textit{\roman{*}}]
\item \label{lem:mu=00003D1/2:part1} $\beta=\tau \min \{\frac{1}{\beta_{l}},\frac{1}{\beta_{h}}\}$ with $\tau$ defined in \eqref{eq:-42-42} and
\begin{equation} \label{eq:-42-42-42-42}
\gamma_{1}^{-1}-\gamma_{2}\|L\|^{2}>\frac{1}{2\beta\gamma_{1}}.
\end{equation}
\item \label{lem:mu=00003D1/2:part2} $l=\iota_{\{0\}}$ and  
\begin{align} \label{eq:-1-1-1}
\gamma_{1}^{-1}-\gamma_{2}\|L\|^{2}>\frac{\beta_{h}}{2}. 
\end{align}
\item \label{lem:mu=00003D1/2:part3} $l=\iota_{\{0\}}$, $h\equiv0$ and 
$\gamma_{1}^{-1}-\gamma_{2}\|L\|^{2}>0$.
\end{enumerate}
Then there exists a pair of solutions
$(x^{\star},y^{\star})\in\mathcal{K}$ to 
\eqref{eq:primal-dual} such that the sequences
$(x_{n})_{n\in\Nn}$ and $(y_{n})_{n\in\Nn}$ converge weakly to $x^{\star}$ and $y^{\star}$, respectively.
\end{prop}
\begin{proof}
See \hyperref[proof of lem:mu=00003D1/2]{Appendix.}
\end{proof}
\begin{rem} \label{rem:-1}
We can derive $o(1/(n+1))$ convergence rate for~\Cref{Algorithm-4} using \Cref{Thm: conv-rates}. It must be noted that this leads to more restrictive conditions compared to the ones required for convergence in~\Cref{lem:mu=00003D1/2}. It was shown in the proof of \Cref{lem:mu=00003D1/2} that $(\lambda_n(\delta-\lambda_n))_{n\in\Nn}\subseteq[\underline{\tau},\infty[$ for some $\underline{\tau}>0$. Furthermore, $P$ and $D$ are given by~\eqref{eq:-17-1} and~\eqref{eqn:D} with $\theta=0$ and $\mu=\frac{1}{2}$, and it's easy to verify that~\eqref{eqnn:-12} holds for $c_1=1$ and $c_2=1+\gamma_1\gamma_2\|L\|^2$.
According to~\Cref{Thm: conv-rates}, the $o(1/(n+1))$ convergence rate is achieved if   $(\lambda_{n})_{n\in\Nn}\subseteq\left[0,c_1\delta/c_2\right]$. 
By substituting $c_1$ and $c_2$, we derive the following sufficient condition
\begin{equation} \label{eqnn:-36}
(1+\gamma_{1}\gamma_{2}\|L\|^{2})^{2}\leq2-\tfrac{1}{2\beta},
\end{equation}
where $\beta$ is defined in~\eqref{eq:-42-43}, and in the special case when $l=\iota_{\{0\}}$, it can be replaced by the simpler term in~\eqref{eq:-30}.
\end{rem}
\begin{rem}
In \Cref{Algorithm-4}, if $\gamma_{1}=\gamma_{2}=\gamma$, $h\equiv0$ and $l=\iota_{\{0\}}$, the error-free version of the algorithm proposed by Brice{\~n}o-Arias and Combettes  \cite[Algorithm (4.8)]{briceno2011monotone+} is recovered. Furthermore, it follows from~\eqref{eqnn:-36} that $o(1/(n+1))$ convergence rate is guaranteed if 
\[
\gamma^2\|L\|^2\leq {\sqrt{2}-1}.
\] 
It is worth mentioning that another splitting was proposed by Combettes and Pesquet in \cite[Theorem 4.2]{combettes2012primal} that solves the minimization problem~\eqref{eq:-21-1}. The aforementioned algorithm requires two evaluations $\nabla h$ and $\nabla l^*$ per iteration. 
\end{rem}

\subsection{\label{sub:theta=00003D1}Special Case \texorpdfstring{$\theta=1$, $\mu=1$}{}}
As another special case, we derive an algorithm that generalizes the scheme proposed recently by Drori and Sabach and Teboulle in~\cite{drori2015simple}. 
When $\theta=1$, \Cref{Algorithm-2} essentially leads to Arrow-Hurwicz updates. 
In addition, we set $\mu=1$, and select $\lambda_n$ in order to have $\alpha_n=1$ (see proof of \Cref{lem:In-Algorithm-of}). This simplifies the iterations and results in \Cref{Algorithm-3}:
\begin{algorithm}[H]
    \caption{}
	\label{Algorithm-3}
    \begin{algorithmic} 
            \item\textbf{Inputs:} $x_0\in\HH$, $y_0\in\mathcal{G}$
            \For{$n=0,\ldots$} 
                \State $\bar{x}_{n} =\prox_{\gamma_{1}f}(x_{n}-\gamma_{1}L^{*}y_{n}-\gamma_{1}\nabla h(x_{n}))$
                \State $y_{n+1} =\prox_{\gamma_{2}g^{*}}(y_{n}+\gamma_{2}L\bar{x}_{n}-\gamma_{2}\nabla l^{*}(y_{n}))$
                \State $x_{n+1} =\bar{x}_{n}-\gamma_{1}L^{*}(y_{n+1}-y_{n})$
            \EndFor
    \end{algorithmic}
    \end{algorithm}
In the next proposition we establish  convergence results for \Cref{Algorithm-3}. The proof is similar to that of \Cref{lem:mu=00003D1/2} and is omitted here.
\begin{prop}
\label{lem:In-Algorithm-of}Consider the sequences $(x_{n})_{n\in\Nn}$ and $(y_{n})_{n\in\Nn}$ generated
by \Cref{Algorithm-3}. Let $\beta_{h}\in]0,+\infty[$ and $\beta_{l}\in]0,+\infty[$ be the
Lipschitz constants of $\nabla h$ and $\nabla l^{*}$, respectively. Suppose that one of the following holds: 
 \begin{enumerate} [{label=(\textit{\roman*})},ref=\textit{\roman{*}}]
 \item \label{lem:In-Algorithm-of:part1}$\beta=\tau \min \{\frac{1}{\beta_{l}},\frac{1}{\beta_{h}}\}$ where $\tau$ is defined in \eqref{eq:-42-42} and
\begin{enumerate} [{label=(\textit{\alph*})},ref=\textit{\alph{*}}]
\item \label{ass:-8-1}  $2\beta>1$,
\item \label{ass:-2-3} $\gamma_{1}^{-1}-\gamma_{2}\left(1+\frac{1}{2(2\beta-1)}\right)^{2}\|L\|^{2}>0$.
\end{enumerate}
\item \label{lem:In-Algorithm-of:part2} $l=\iota_{\{0\}}$ and 
 \begin{equation} \beta_{h}\gamma_{1}<2-\gamma_{1}\gamma_{2}\|L\|^{2}-\sqrt{\gamma_{1}\gamma_{2}\|L\|^{2}}. \label{eq:-42-39}
 \end{equation}
 \item \label{lem:In-Algorithm-of:part3} $l=\iota_{\{0\}}$, $h\equiv0$ and 
 \begin{equation*} 
\gamma_{1}^{-1}-\gamma_{2}\|L\|^{2}>0.
 \end{equation*}
\end{enumerate}
 Then there exists a pair of solutions
$(x^{\star},y^{\star})\in\mathcal{K}$ to 
\eqref{eq:primal-dual} such that the sequences
$(x_{n})_{n\in\Nn}$ and $(y_{n})_{n\in\Nn}$ converge weakly to $x^{\star}$ and $y^{\star}$, respectively.
\end{prop}
\begin{rem}
In \Cref{rem:-1} we deduced $o(1/(n+1))$ convergence rate for~\Cref{Algorithm-4} using \Cref{Thm: conv-rates}. We can similarly deduce the same rate for~\Cref{Algorithm-3}. However, also in this case \Cref{Thm: conv-rates} imposes more restrictive conditions compared to the ones required for convergence in~\Cref{lem:In-Algorithm-of}. 
It is easy to verify that in this case~\eqref{eqnn:-12} holds for 
$c_1=2/(2+\sqrt{\gamma_{1}\gamma_{2}\|L\|^{2}})$, $c_2=2/(2-\sqrt{\gamma_{1}\gamma_{2}\|L\|^{2}})$.
According to \Cref{Thm: conv-rates} we have $o(1/(n+1))$ convergence rate if  $(\lambda_{n})_{n\in\Nn}\subseteq\left[0,c_1\delta/c_2\right]$. With $c_1$ and $c_2$ defined as above we derive the following sufficient condition
\begin{equation} \label{eqnn:-tebl-conv1}
\frac{1}{\beta}<4-\tfrac{4\left(2+\sqrt{\gamma_1\gamma_2\|L\|^2}\right)}{\left(2-\sqrt{\gamma_1\gamma_2\|L\|^2}\right)^2},
\end{equation}
with $\beta=\tau \min \{\frac{1}{\beta_{l}},\frac{1}{\beta_{h}}\}$, where $\tau$ is defined in \eqref{eq:-42-42}. For $l=\iota_{\{0\}}$,~\eqref{eqnn:-tebl-conv1} becomes
\begin{equation*} 
	\beta_h\gamma_1<4-\gamma_1\gamma_2\|L\|^2-\tfrac{\left(2+\sqrt{\gamma_1\gamma_2\|L\|^2}\right)^3}{4-\gamma_1\gamma_2\|L\|^2}.
\end{equation*}
If in addition $h\equiv 0$, 
it further simplifies to  $\gamma_1\gamma_2\|L\|^2< \frac{21-5\sqrt{17}}{2}=0.192$.
\end{rem}
\begin{rem}
	\Cref{Algorithm-3} with $f\equiv^{}0$ and $l=\iota_{\{0\}}$ reduces to the algorithm proposed by Drori, Sabach and Teboulle in~\cite{drori2015simple}. The authors require  
	$\beta_{h}\gamma_{1}\leq1$, $\gamma_{1}\gamma_{2}\|L\|^{2}\leq1$,
	which is less restrictive that~\eqref{eq:-42-39}. 
\end{rem}

\subsection{Special Case: \texorpdfstring{$\mu=0$}{}} \label{DR-ADMM}
In this section we discuss three algorithms that are all special cases of  \Cref{Algorithm-2} when $\mu=0$.  \Cref{Algorithm-6b} was already discovered in \Cref{sec:special cases}. Here we briefly restate it in the optimization framework. \Cref{Algorithm-7} is an application of \Cref{Algorithm-6b} to a dual problem formulation.  \Cref{Algorithm-5} is derived from \Cref{Algorithm-2} when $\mu=0$, $h\equiv0$ and $l=\iota_{\{0\}}$, by selecting $\lambda_n$ in order to have constant stepsize $\alpha_n=1$. A more general algorithm can be considered that includes $h$ and $l$, but this would complicate the conditions for selecting the parameters.
\begin{algorithm}[H]
    \caption{}
	\label{Algorithm-5}
    \begin{algorithmic} 
            \item\textbf{Inputs:} $x_0\in\HH$, $y_0\in\mathcal{G}$
            \For{$n=0,\ldots$} 
                \State $x_{n+1}  =\prox_{\gamma_{1}f}(x_{n}-\gamma_{1}L^{*}y_{n})$
                \State $\bar{y}_{n} =\prox_{\gamma_{2}g^{*}}(y_{n}+\gamma_{2}L((1-\theta)x_{n}+\theta x_{n+1}))$
                \State $y_{n+1} =\bar{y}_{n}+\gamma_{2}(2-\theta)L(x_{n+1}-x_{n})$
            \EndFor
    \end{algorithmic}
\end{algorithm}
The next proposition provides the convergence results for~\Cref{Algorithm-5} through the same reasoning used in~\Cref{lem:In-Algorithm-of,lem:mu=00003D1/2}.

\begin{prop}
	\label{lem:Let--be} Assume that $h\equiv 0$ and $l=\iota_{\{0\}}$. Consider the sequences $(x_{n})_{n\in\Nn}$ and $(y_{n})_{n\in\Nn}$ generated
	by \Cref{Algorithm-5}. Let $\theta\in[0,\infty[$ and suppose that one of the following holds:
	\begin{enumerate}[{label=(\textit{\roman*})},ref=\textit{\roman{*}}]
		\item \label{prop4:part1} 
		$\gamma_{1}^{-1}-\gamma_{2}(\theta^{2}-3\theta+3)\|L\|^{2}>0$,
		\item $\theta=2$,  $\mathcal{K}$ is finite-dimensional and
		$\gamma_{1}^{-1}-\gamma_{2}\|L\|^{2}\geq0$.   
	\end{enumerate}
	Then there exists a pair of solutions $(x^{\star},y^{\star})\in\mathcal{K}$ to 
	\eqref{eq:primal-dual} such that the sequences $(x_{n})_{n\in\Nn}$ and $(y_{n})_{n\in\Nn}$ converge weakly to $x^{\star}$ and $y^{\star}$, respectively.
\end{prop}
\begin{rem}
\Cref{Algorithm-6} was introduced in \Cref{sec:special cases} for the monotone inclusion problem~\eqref{eq:DR main}. It can be rediscovered here by considering \Cref{Algorithm-2} with $\mu=0$ and setting $l=\iota_{\{0\}}$, $L$ to be the identity and $\gamma_{1}=\gamma_{2}^{-1}=\gamma$ and following the same algebraic manipulations. For this reason we only state it in the framework of optimization and refer the reader to \Cref{prop:-1} for convergence results. 
\end{rem}
\begin{algorithm}[H]
    \caption{Douglas-Rachford Type with a Forward Term}
	\label{Algorithm-6b}
    \begin{algorithmic} 
            \item\textbf{Inputs:} $x_0\in\HH$, $s_0\in\HH$
            \For{$n=0,\ldots$} 
                \State $\bar{x}_{n} =\prox_{\gamma f}(s_{n}-\gamma\nabla h(x_{n}))$
                \State $r_{n} = \prox_{\gamma g}(\theta \bar{x}_{n}+(2-\theta)x_{n}-s_{n})$
                \State $s_{n+1} =s_{n}+\rho_n(r_{n}-\bar{x}_{n})$
                \State $x_{n+1}=x_n+\rho_n(\bar{x}_n-x_n)$
            \EndFor
    \end{algorithmic}
\end{algorithm}

\subsection*{ADMM Form} 
Consider the following problem 
\begin{subequations} \label{ADMM}
\begin{align}
\underset{x_1,x_2,x_3}{\minimize}\ & \quad f_1(x_1)+f_2(x_2)+f_3(x_3) \\
\stt & \quad L_1 x_1+L_2 x_2 +L_3 x_3 = b,
\end{align}
\end{subequations}
 where $\HH_1,\dots,\HH_4$ are real Hilbert spaces, $f_i\in\Gamma_{0}(\HH_i)$,
$L_i\in\mathscr{B}(\HH_i,\mathcal{\HH}_4)$ for $i=1,2,3$ and $b\in \HH_4$. Additionally, $f_1$
is $\xi$-strongly convex for some $\xi\in]0,+\infty[$.
It is well known that the classic Alternating Direction
Method of Multipliers
(ADMM) is equivalent to  Douglas-Rachford algorithm applied to the dual problem (see \cite{boyd2011distributed} and the references therein). We derive a new 3-block ADMM iteration in a similar way. Consider the dual problem 
\begin{equation}
\underset{y}{\minimize}\ d_1(y)+d_2(y)+d_3(y), \label{dual}
\end{equation}
where $d_1(y)=f^*_{1}(-L^*_{1} y)$, $d_2(y)=f^*_{2}(-L^*_{2} y)$ and $d_3(y)=f^*_{3}(-L^*_{3} y)+\langle y,b \rangle$.  By strong convexity of $f_1$ we have that $d_1$ has a $\xi^{-1}\|L_1\|^2$-Lipschitz gradient. Form the augmented Lagrangian 
\[ 
\mathcal{L}_{\gamma}(x_{1},x_{2},x_{3},y)=\sum_{i=1}^{3}f_{i}(x_{i})+\left\langle y,\sum_{i=1}^{3}L_{i}x_{i}-b\right\rangle+\frac{\gamma}{2}\left\|\sum_{i=1}^{3}L_{i}x_{i}-b\right\|^{2}.
\]
 We apply \Cref{Algorithm-6b} with $\rho_n=1$ to \eqref{dual} and after a change of order and some algebraic manipulations we derive \Cref{Algorithm-7} (the procedure is similar to the one found in \cite[Section 3.5.6]{eckstein1992douglas} for the classic ADMM).  Our 3-block ADMM can be written as
\begin{algorithm}[H]
    \caption{3-block ADMM} \label{Algorithm-7}
    \begin{algorithmic} 
            \item\textbf{Inputs:} $(x_{1,0},x_{2,0},x_{3,0})\in \HH_1 \times \HH_2 \times \HH_3$, $(y_0,y_1)\in\HH_4 \times \HH_4$
            \For{$n=0,\ldots$} 
                \State $\bar{y}_{n}=(\theta-1)y_{n}+(2-\theta)y_{n-1}$
                \State $x_{1,n+1}=\argmin_{x_{1}}\mathcal{L}_{0}(x_{1},x_{2,n},x_{3,n},y_{n})$
                \State $x_{2,n+1}= \argmin_{x_{2}}\mathcal{L}_{\gamma}(x_{1,n},x_{2},x_{3,n},\bar{y}_{n})$
                \State $x_{3,n+1}= \argmin_{x_{3}}\mathcal{L}_{\gamma}(x_{1,n+1},x_{2,n+1},x_{3},\bar{y}_{n})$
                \State ${y}_{n+1}=\bar{y}_{n}+\gamma(L_{1}x_{1,n+1}+L_{2}x_{2,n+1}+L_{3}x_{3,n+1}-b)$
            \EndFor
    \end{algorithmic}
\end{algorithm}
\begin{prop} \label{prop:9}
Let $\HH_1,\dots,\HH_4$ be finite-dimensional,  $f_i\in\Gamma_{0}(\HH_i)$, $b\in \HH_4$, 
$L_i\in\mathscr{B}(\HH_i,\mathcal{\HH}_4)$ for $i=1,2,3$ and $\ker(L_2)=\{0\}$, $\ker(L_3)=\{0\}$. 
Let $\xi\in]0,+\infty[$ be the strong convexity constant of $f_1$. 
Assume that the set of saddle points of~\eqref{ADMM}, denoted by $\Sigma$, is nonempty. Let $\theta\in]1,2[$ and  
\begin{equation}
{\gamma} < \xi(2-\theta)(\theta-\sqrt{2-\theta})/\|L_1\|^2. \label{eqnn:-4}
\end{equation}
  Then the sequence $(x_{1,n},x_{2,n},x_{3,n},y_{n})_{n\in\mathbb{N}}$ generated by \Cref{Algorithm-7},
  converges to some  $(x_{1}^{\star},x_{2}^{\star},x_{3}^{\star},y^{\star})\in\Sigma$. 
  \end{prop}
  \begin{proof}
  	Let $(x_{1}^{\star},x_{2}^{\star},x_{3}^{\star},y^{\star})$ denote a KKT point of~\eqref{ADMM}, \ie, 
  	\begin{equation} \label{eqnn:-5} 
  	\begin{cases}
  	0\in\partial f_{i}(x_{i}^\star)+L_{i}^{*}y^\star, &\textrm{for}\;\, i=1,2,3\\
  	0 = b-L_{1}x_{1}^\star-L_{2}x_{2}^\star-L_{3}x_{3}^\star.&
  	\end{cases} 
  	\end{equation}
  	\Cref{Algorithm-7} is an implementation of \Cref{Algorithm-6b} (or \Cref{Algorithm-6} in the framework of general monotone operators) when applied to the dual problem~\eqref{dual}. Hence, \Cref{prop:-1} yields $y_n\rightarrow y$, where $y$ is a solution to~\eqref{dual}. Let $x_1$ be a point satisfying $-L_1^*y\in \partial f_1(x_1)$. 	
  	From strong convexity of $f_1$ at $x_1$ and $x_{1,n+1}$, we have
  	\begin{equation*}
  		(\forall u\in\partial f_1({x}_{1,n+1}))(\forall v\in\partial f_1({x}_1))\quad\xi\|x_{1,n+1}-{x}_{1}\|^{2}\leq\langle u-v,x_{1,n+1}-{x}_{1}\rangle.
  	\end{equation*}
  	It follows from the optimality condition for the $x_{1,n+1}$ update and the definition of $x_1$ that
  	\begin{equation*} 
  		\xi\|x_{1,n+1}-{x}_{1}\|^{2}\leq    \langle -L_{1}^{*}y_{n}+L_{1}^{*}{y},x_{1,n+1}-{x}_{1}\rangle \leq \|L_1\|\|x_{1,n+1}-{x}_{1}\|\|y_{n}-{y}\|.
  	\end{equation*}
  	Combine this with the convergence of $(y_n)_{n\in\Nn}$ to derive $x_{1,n}\to x_1$.  From the change of variables to derive \Cref{Algorithm-7} we have
  	\begin{equation} \label{eq:-106}
  	s_{n}-y_{n}=-\gamma L_{1}x_{1,n}-\gamma L_{3}x_{3,n},
  	\end{equation}
  	which together with the convergence of $(s_n)_{n\in\Nn},(y_n)_{n\in\Nn}$ (see \Cref{prop:-1}) and $(x_{1,{n}})_{n\in\Nn}$ imply that $(L_{3}x_{3,n})_{n\in\Nn}$ converges to a point. Since $\ker(L_3)=\{0\}$, it follows that $(x_{3,{n}})_{n\in\Nn}$ converges to some $x_3$. From the optimality condition for the $x_{3,n+1}$ update and the last step in \Cref{Algorithm-7}, we have
  	\begin{equation*}
  		-L_3^*y_{n+1}=-L_3^*\bar{y}_n-\gamma L_3^*(L_{1}x_{1,n+1}+L_{2}x_{2,n+1}+L_{3}x_{3,n+1}-b)\in\partial f_3(x_{3,n+1}). 
  	\end{equation*}
  	Taking the limit and using  \cite[Proposition 20.33(iii)]{bauschke2011convex}, we have $-L_3^*y\in \partial f_3(x_3)$. On the other hand,  \Cref{thm:Suppose-that-}\eqref{thm-part2} and the last line of \Cref{Algorithm-7} yield
  	\begin{equation} \label{eq:-103}
  	y_{n+1}-y_n\rightarrow 0, \quad \textrm{and} \quad L_{1}x_{1,n}+L_{2}x_{2,n}+L_{3}x_{3,n}\to b. 
  	\end{equation}
  	It follows from~\eqref{eq:-106},~\eqref{eq:-103}, and  the convergence of $(s_n)_{n\in\Nn},(y_n)_{n\in\Nn}$ that $(L_{2}x_{2,n})_{n\in\Nn}$ converges to a point. 
  	We can now argue almost exactly as we did for $(L_3x_{3,n})_{n\in\Nn}$. Since $\ker(L_2)=\{0\}$, we deduce that $(x_{2,{n}})_{n\in\Nn}$ converges to some $x_2$. Combine the optimality condition for the $x_{2,n+1}$ update and the last step in \Cref{Algorithm-7} with the convergence of $y_n$, to derive   $-L_2^*y\in \partial f_2(x_2)$. Altogether, we showed that the limit points $(x_1,x_2,x_3,y)$, are jointly optimal by the KKT condition~\eqref{eqnn:-5}.
  	\end{proof}
\begin{rem}
	The convergence rate of \Cref{Algorithm-7} can be deduced similar to \Cref{Algorithm-6} from \Cref{Thm: conv-rates,Thm: conv-rates-ii,thm:Suppose-that--1} with $\rho_n=1$. However, we do not consider it in this paper due to lack of space. 
\end{rem}

\begin{rem}
	In the case when $f_1\equiv 0$ we can choose the limiting value $\theta=2$ and recover the classical ADMM (see \Cref{prop:-1}\eqref{prop:-1-part3}). On the other hand if $f_2$ or $f_3$ vanish, the Alternating Minimization Method (AMM)~\cite{tseng1991applications} is recovered. Finally, when both $f_2$ and $f_3$ vanish then the dual ascent method is recovered. 
\end{rem}
\begin{rem}
	In the recent work \cite[Algorithm 8]{davis2015three}, another 3-block ADMM formulation is presented by following similar algebraic manipulations (It is derived by applying their Algorithm 7 to the dual). It should be noted that they do not require rank assumptions on $L_2,L_3$. 
	In contrast to that work in our version $(x_{1,n})_{n\in\Nn}$ and $(x_{2,n})_{n\in\Nn}$ are  updated in parallel which corresponds to the fact that in \Cref{Algorithm-6} the forward step precedes the first prox step. Furthermore, in our algorithm, $(x_{2,n})_{n\in\Nn}$ and $(x_{3,n})_{n\in\Nn}$ are updated using the augmented Lagrangian at $(\theta-1)y_{n}+(2-\theta)y_{n-1}$ rather than $y_n$. Moreover, the stepsize in \cite[Theorem 2.1]{davis2015three} has to satisfy $\gamma<2\xi /\|L_1\|^2$. This is always larger than the stepsize in~\eqref{eqnn:-4}. Refer to~\Cref{rem:DR} for further discussion, noting that \Cref{Algorithm-7} is derived by setting the relaxation parameter, $\rho_n$, equal to one. 
\end{rem}

\begin{rem}
	Some of the other recent attempts to directly generalize ADMM for 3 blocks include \cite{cai2014direct,chen2014direct,han2012note,li2015convergent,lin2015iteration}. In \cite{chen2014direct}, it was shown through a counterexample that a direct extension of ADMM to more that 2 blocks is not convergent in general. In order to ensure convergence, additional assumptions on strong convexity of the functions or rank of $L_i$'s are needed. 
	In \cite{cai2014direct} the authors require one function to be strongly convex and $L_2$ and $L_3$ to have full column rank, while \cite{li2015convergent} modify the steps with regularization terms and \cite{lin2015iteration} solves a perturbed problem (see \cite{lin2015iteration} and the references therein for further discussion). In contrast to these papers, the first minimization step of \Cref{Algorithm-7} consists of minimizing a normal Lagrangian rather than an augmented one (therefore it can be trivially executed in a distributed fashion in the case where $f_1$ is block-separable) and it can be performed in parallel to the second step.
\end{rem}


\section{Conclusion and Future Work}

In this paper we introduced a unifying  operator splitting technique called \linebreak Asymmetric-Forward-Backward-Adjoint splitting (AFBA) for solving monotone inclusions involving three terms. We discussed how it relates to, unifies and extends classical splitting methods as well as several primal-dual algorithms for structured nonsmooth convex optimization. 
Asymmetric preconditioning is the main feature of AFBA that can lead to several extensions and new algorithmic schemes. Exploring block triangular choices for the operator $H$ in \eqref{eq:-42-42-42} that lead to Gauss-Seidel-type updates for multi-block structured convex optimization problems is certainly a promising research direction. Another research direction involves investigating if the linear operator $M$ can be replaced by a Lipschitz operator.  It is also desirable to study randomized variants of our algorithm for problems with block structure in order to further reduce memory and computational requirements. Some natural applications include distributed optimization, control and image processing.
\bibliographystyle{amsplain}
\bibliography{AFBA_revised}
\renewenvironment{thm}[1][]{\begin{appendixthm}[#1]}{\end{appendixthm}}
\renewenvironment{prop}[1][]{\begin{appendixprop}[#1]}{\end{appendixprop}}
\section{Appendix}
The proofs omitted in the main text are listed below. 
\begin{Proof}{Proof of \Cref{prop:-1}} \label{proof of prop:-1}
\eqref{prop:-1-part1}:
Let us start by noting that $P$ and $S$ defined in~\eqref{eq:-42-all} are special cases of~\eqref{eq:-17-1} and~\eqref{eqn:S} when $L=\id$, $\gamma_1=\gamma_2^{-1}=\gamma$ and, as such, are strongly positive if $\theta\in[0,2[$ (see \Cref{lem:strongly positive p-1,lem:strongly}). 
Next, consider \hyperlink{alg:line3}{step 3} of \Cref{Algorithm-1} and Substitute the parameters defined in~\eqref{eq:-42-all}, to derive 
	 \begin{align}\frac{\alpha_{n}}{\lambda_n} & = \frac{\|\tilde{z}_{n}\|_{P}^{2}}{\|\tilde{z}_{n}\|_{D}^{2}} = \frac{\gamma^{-1}\|\tilde{x}_{n}\|^{2}+\gamma\|\tilde{y}_{n}\|^{2}-\theta\langle\tilde{x}_{n},\tilde{y}_{n}\rangle}{\gamma^{-1}(\theta^2-3\theta+3)\|\tilde{x}_n\|^2+\gamma\|\tilde{y}_n\|^2+2(1-\theta)\langle \tilde{x}_n,\tilde{y}_n \rangle}, \label{eq:prop1}\end{align}
	 where $D$ is defined in~\eqref{eq:D} and by construction is strongly positive. 
 Let $\lambda_n$ be equal to the inverse of the right hand side in~\eqref{eq:prop1} multiplied by $\rho_n$ where $(\rho_n)_{n\in\Nn}\subseteq]0,2[$. This would result in $\alpha_n=\rho_n$ and simplify the iterations. Consider $C$ defined in~\eqref{eq:Copt1}, we have
\[
\|Cz-Cz^{\prime}\|^2_{P^{-1}} = \gamma\left(1-\tfrac{1}{4}\theta^2\right)^{-1}\|Cz-Cz^{\prime}\|^2.
\] 
From here it is easy to see that $C$ is $\beta$-cocoercive with respect to $P$ norm, where $\beta=\eta\gamma^{-1}(1-\frac{1}{4}\theta^2)$.  
 In order to apply \Cref{thm:Suppose-that-},  condition~\eqref{eq:-7} must hold. From strong positivity of $D$, boundedness of $P$ and the fact that $\rho_n$ is uniformly bounded above $0$ it follows that $(\lambda_{n})_{n\in\Nn}\subseteq]\nu_1,\infty[$ for some positive $\nu_1$. Let $\nu_2$ be a positive parameter such that
 \begin{equation} \label{eqn:-22}
 \nu_2<2-\frac{1}{2\beta}-\frac{2\rho_n(2+\sqrt{2-\theta})}{2+\theta},
 \end{equation} 
 holds for all $n\in\Nn$. 
 It's easy to verify that such $\nu_2$ exists as long as $\rho_n$ is uniformly bounded in the interval~\eqref{eq:rho}. For brevity, we define $\beta^\prime$ such that $\frac{1}{2\beta^\prime}=\frac{1}{2\beta}+\nu_2$. Additionally, 
 introduce the notation $\upsilon=\rho_n^{-1}(2-\frac{1}{2\beta^\prime})$,
 $\omega_{1}=\upsilon\theta+2(1-\theta)$, $\omega_{2}=\theta^2-3\theta+3$. 
We proceed by showing that $\lambda_{n}$ is smaller than $2-\frac{1}{2\beta}-\nu_2$. A sufficient condition for this to hold is
\begin{equation} \label{eq:-544}
\xi=\gamma^{-1}(\upsilon-\omega_2)\|\tilde{x}_{n}\|^{2}+\gamma(\upsilon-1)\|\tilde{y}_{n}\|^{2}-\omega_{1}\langle\tilde{x}_{n},\tilde{y}_{n}\rangle>0.
\end{equation}
Apply the Fenchel-Young inequality for $\frac{\epsilon}{2}\|\cdot\|^2$ to lower bound $\xi$:
\begin{equation}
\xi \geq  \left(\gamma^{-1}(\upsilon-\omega_2)-|\omega_{1}|\tfrac{\epsilon}{2}\right)\|\tilde{x}_{n}\|^{2}+\left(\gamma(\upsilon-1)-|\omega_1|\tfrac{1}{2\epsilon}\right)\|\tilde{y}_{n}\|^{2}, \label{eq:-42-n}
\end{equation}
where $|\cdot|$ is the absolute value. It follows from~\eqref{eqn:-22} that $\upsilon>1$ for all $n\in\Nn$.  Let $\epsilon=\frac{|\omega_{1}|}{2\gamma(\upsilon-1)}$ so that the term involving $\|\tilde{y}_n\|^2$ in~\eqref{eq:-42-n} disappears. We obtain
\begin{align}
& \gamma^{-1}\left(\upsilon-\omega_2-\tfrac{\omega_1^2}{4(\upsilon-1)}\right)\|\tilde{x}_{n}\|^{2}>0, \label{eq:-42-nn}
\end{align}
which is sufficient for~\eqref{eq:-544} to hold. By substituting $\omega_{1}$ and $\omega_{2}$ and after some algebraic manipulations we find that the condition~\eqref{eqn:-22} is sufficient for the right hand side in~\eqref{eq:-42-nn} to be positive. Consequently,  \Cref{thm:Suppose-that-} completes the proof of convergence. We showed that we can set $\alpha_n=\rho_n$ by choosing $\lambda_n$ appropriately. 
\Cref{Algorithm-6} follows by setting  $\alpha_n=\rho_n$, a change of variables $s_n=x_n-\gamma y_n$, substituting $x_{n+1}$ and application of Moreau's identity.  

\eqref{prop:-1-part2}: Mimic the proof of~\eqref{prop:-1-part1}, but use \Cref{thm:Suppose-that-} with $C\equiv0$, by showing that $\lambda_n$ is uniformly bounded between $0$ and $2$. 

\eqref{prop:-1-part3}:  When $\theta=2$, we have $P\in \mathcal{S}(\mathcal{K})$, $P\succeq0$. It follows from \eqref{eq:Aopt1},~\eqref{eq:-42-H} and \Cref{lem:separable} that $(H+A)^{-1}P$ is continuous. Therefore, since $F\equiv0$, by appealing to  \Cref{thm:Suppose-that--1} and following the same change of variables as in previous parts the assertion is proved. 
\end{Proof}
\begin{Proof}{Proof of \Cref{porp:-1-conv}} \label{proof of prop:-1-conv}
Following the argument in proof of \Cref{prop:-1}\eqref{prop:-1-part3} and \Cref{thm:Suppose-that--1}\eqref{thm-P-part4} yields
\begin{equation}\label{eq:-100}
 \|{P}z_{n+1}-Pz_n\|^{2}\leq\tfrac{\|P\|}{\tau(n+1)}\|Qz_{0}-Qz^{\star}\|_{R}^{2},
\end{equation}
 and $\|Pz_{n+1}-Pz_n\|^{2}=o(1/(n+1))$, where $z_n= \left(x_n,\gamma^{-1}(x_n-s_n)\right)$. Combine this with definition of $P$, ~\eqref{eq:-42-p} with $\theta=2$, to derive
 \begin{equation} \label{eq:-101}
 \|{P}z_{n+1}-Pz_n\|^{2}= (1+\gamma^{-2})\|s_{n+1}-s_n\|^2.
 \end{equation}
 Furthermore, simple calculation shows that $P^2=(\gamma+\gamma^{-1})P$. Hence for all $z\in\mathcal{K}$
 \begin{align}
\|Pz\|^2&=(\gamma+\gamma^{-1})\langle z,Pz\rangle \nonumber\\
&=  (\gamma+\gamma^{-1})\left(\gamma^{-1}\|x\|^2+\gamma\|y\|^2-2\langle x,y \rangle\right)\nonumber\\
& \leq (\gamma+\gamma^{-1})^2\|z\|^2,\label{eq:pnorm}
 \end{align}
 where we used Fenchel-Young inequality with $\epsilon=\gamma$. It follows from~\eqref{eq:pnorm} that $\|P\|\leq(\gamma+\gamma^{-1})$. 
 Combining this with~\eqref{eq:-100} and~\eqref{eq:-101} completes the proof.
\end{Proof}

\begin{Proof}{Proof of \Cref{lem: cocoercivity lemma}}
\label{proof of lem: cocoercivity lemma}
For every $z$ and $z_1$ in $\mathcal{K}$ we have 
\begin{align}
	\langle Cz-Cz^{\prime},z-z^{\prime}\rangle & = \langle \nabla h(x)-\nabla h(x^{\prime}),x-x^\prime \rangle + \langle \nabla l^*(y)-\nabla l^*(y^{\prime}),y-y^\prime \rangle \nonumber \\
	&
	\geq \frac{1}{\beta_{h}}\|\nabla h(x)-\nabla h(x^{\prime})\|^{2} +\frac{1}{\beta_{l}}\|\nabla l^*(y)-\nabla l^*(y^{\prime})\|^{2} \nonumber \\
	& 
	\geq \min \{\beta_{l}^{-1},\beta_{h}^{-1}\}\| Cz-Cz^{\prime} \|^{2}. \label{eq:-12-32}
\end{align}
It follows from \Cref{lem:strongly positive p-1} that $P\in\mathcal{S}_\tau(\mathcal{K})$ with $\tau$ defined in~\eqref{eq:-42-42}. Therefore, we have $\|P^{-1}\|\leq\tau^{-1}$.
Using \eqref{eq:-12-32}, we have
\begin{align*}
	\langle Cz-Cz^{\prime},z-z^{\prime}\rangle & \geq \min \{\beta_{l}^{-1},\beta_{h}^{-1}\}\| Cz-Cz^{\prime} \|^{2} \\
	& 
	\geq \|P^{-1}\|^{-1}\min \{\beta_{l}^{-1},\beta_{h}^{-1}\}\| Cz-Cz^{\prime} \|_{P^{-1}}^{2} \\
	& \geq \tau\min \{\beta_{l}^{-1},\beta_{h}^{-1}\}\| Cz-Cz^{\prime} \|_{P^{-1}}^{2}.
\end{align*}
Hence, $C$ is $\beta$-cocoercive with respect to $\|\cdot\|_P$, where $\beta=\tau\min \{\beta_{l}^{-1},\beta_{h}^{-1}\}$.
In the case when $l=\iota_{\{0\}}$, we have $\nabla l^{*}=0$ and it is possible to derive less conservative parameters. Define the linear
operator $Q:(x,y)\rightarrow(x,0)$, then, 
\begin{align*}
	\|Qz\|_{P^{-1}}^{2} & =\langle x,(\gamma_{1}^{-1}\id-\frac{\gamma_{2}}{4}\theta^{2}L^{*}L)^{-1}x\rangle \leq\|x\|^{2}\|(\gamma_{1}^{-1}\id-\frac{\gamma_{2}}{4}\theta^{2}L^{*}L)^{-1}\| 
	\\
	& \leq\|Qz\|^{2}(\gamma_{1}^{-1}-\frac{\gamma_{2}}{4}\theta^{2}\|L\|^{2})^{-1},
\end{align*}
where we used \eqref{eq:-17-1}. 
For $\beta_{h}\in]0,+\infty[$, considering the above inequality and the fact that $Cz=(\nabla h(x),0)$
we obtain
\begin{align*}
	\|Cz-Cz^{\prime}\|_{P^{-1}}^{2} & \leq(\gamma_{1}^{-1}-\frac{\gamma_{2}}{4}\theta^{2}\|L\|^{2})^{-1}\|\nabla h(x)-\nabla h(x^{\prime})\|^{2}\\
	& \leq\beta_{h}(\gamma_{1}^{-1}-\frac{\gamma_{2}}{4}\theta^{2}\|L\|^{2})^{-1}\langle \nabla h(x)-\nabla h(x^{\prime}),x-x^\prime \rangle\\
	& =\beta_{h}(\gamma_{1}^{-1}-\frac{\gamma_{2}}{4}\theta^{2}\|L\|^{2})^{-1}\langle Cz-Cz^{\prime},z-z^\prime \rangle,
\end{align*}
where the last inequality follows from the Baillon-Haddad theorem \cite[Corollary 18.16]{bauschke2011convex}. 
This shows that $C$ is $\beta$-cocoercive
with respect to $\|\cdot\|_P$, with $\beta$ defined in~\eqref{eq:-30}.
\end{Proof}
\begin{Proof}{Proof of \Cref{lem:strongly}}
\label{proof of lem:strongly} 
 We start by showing that $S_2$ is strongly positive. \Cref{lem:strongly positive p-1} asserts that $P\in\mathcal{S}_\tau(\mathcal{K})$ with $\tau$ defined in~\eqref{eq:-42-42}. Simple algebra shows that $\tau\leq \min\{\gamma_1^{-1},\gamma_2^{-1}\}$. If $\theta\in[2,\infty[$  define a constant $c$, such that $0<c<\tau$, otherwise if $\theta\in[0,2[$ set $c$ such that $0<c<\min\{\tau,\frac{(2-\theta)\gamma_{2}^{-1}}{4}\}$. Then for all $z=(x,y)\in\mathcal{K}$, $\langle (P-c\id)z,z \rangle>0$  which implies
\begin{equation}
(\gamma_{1}^{-1}-c)\|x\|^2-\tfrac{\theta^{2}}{4(\gamma_{2}^{-1}-c)}\|Lx\|^2 > 0.\label{eq:-36}
\end{equation}
It follows from~\eqref{eq:S-6} that
\begin{align*}
\langle z,S_{2}z\rangle & =   \gamma_{1}^{-1}\|x\|^{2}-2\langle Lx,y\rangle+\gamma_{2}^{-1}\|y\|^{2}+\gamma_{2}(2-\theta)\|Lx\|^{2}  \\ &
  =   c\|z\|^{2}+(\gamma_{1}^{-1}-c)\|x\|^{2}-2\langle Lx,y\rangle+(\gamma_{2}^{-1}-c)\|y\|^{2}+\gamma_{2}(2-\theta)\|Lx\|^{2} \\ &
  \geq  c\|z\|^{2}+(\gamma_{1}^{-1}-c)\|x\|^{2}+(\gamma_{2}^{-1}-c-\epsilon)\|y\|^{2}+(\gamma_{2}(2-\theta)-\epsilon^{-1})\|Lx\|^{2}, 
\end{align*}
 where the inequality follows from  Fenchel-Young inequality for $\frac{\epsilon}{2}\|\cdot\|^2$.
Let $\epsilon=\gamma_{2}^{-1}-c$ so that the term involving $\|y\|^2$ disappears, and use \eqref{eq:-36} to derive
\begin{align}
\langle z,S_{2}z\rangle & \geq\ c\|z\|^{2}+(\gamma_{1}^{-1}-c)\|x\|^{2}+\left(\gamma_{2}\left(2-\theta\right)-\tfrac{1}{\gamma_{2}^{-1}-c}\right)\|Lx\|^{2} \nonumber\\
 & > c\|z\|^{2}+\left(\gamma_{2}(2-\theta)+\left(\tfrac{\theta^{2}}{4}-1\right)\tfrac{1}{\gamma_{2}^{-1}-c}\right)\|Lx\|^{2}. \label{eqn:-7}
\end{align}
It follows from the choice of $c$ for both $\theta\in[0,2[$ and $\theta\in[2,\infty[$ that the last term in~\eqref{eqn:-7} is nonnegative and hence, $S_{2}\in\mathcal{S}_c(\mathcal{K})$. The analysis for $S_1$ is similar and have been omitted here. Hence, it follows that $S_1^{-1}$ and $S_2^{-1}$ are strongly positive and so is their convex combination and its inverse, $S$. Furthermore, It was shown in~\eqref{eq:D pos} that if $S$ and $P$ are strongly positive, $K$ skew adjoint and $M$ monotone, then  $D\in\mathcal{S}_\nu(\mathcal{K})$ for some $\nu\in]0,\infty[$, where $D$ was defined in~\eqref{eq:D}. We conclude by noting that $D$ in \eqref{eqn:D} can be derived by substituting $S$ in~\eqref{eq:D}.
\end{Proof}
\begin{proof}[Proof of \Cref{lem:mu=00003D1/2}]
\label{proof of lem:mu=00003D1/2}
  \Cref{Algorithm-4} is a special case of \Cref{Algorithm-2} and consequently of \Cref{Algorithm-1}. Let us establish the proof by directly showing that~\Cref{thm:Suppose-that-} is applicable with $(\lambda_n)_{n\in\Nn}$ selected such that $\alpha_{n}=1$,
  while satisfying~\eqref{eq:-7}. For all \eqref{lem:mu=00003D1/2:part1}-\eqref{lem:mu=00003D1/2:part3},  strong positivity of the operators $P$, $S$, and $D$  follow from \Cref{lem:strongly positive p-1,lem:strongly}. Set
  \begin{align} \label{eqn:-23}
  	\lambda_{n} = \frac{\|\tilde{z}_{n}\|_{D}^{2}}{\|\tilde{z}_{n}\|_{P}^{2}}= 1+\frac{\gamma_{2}\|L\tilde{x}_{n}\|^{2}+\gamma_{1}\|L^{*}\tilde{y}_{n}\|^{2}}{\gamma_{1}^{-1}\|\tilde{x}_{n}\|^{2}+\gamma_{2}^{-1}\|\tilde{y}_{n}\|^{2}}.
  \end{align}
  Such a choice would yield $\alpha_n=1$. A sufficient condition for~\eqref{eq:-7} to hold is to show that $(\lambda_n)_{n\in\Nn}$ is uniformly bounded between $0$ and $\delta$. 
  From strong positivity of $D$ and boundedness of $P$ it follows that $(\lambda_{n})_{n\in\Nn}\subseteq[\nu_1,\infty[$ for some positive $\nu_1$. Denote by $\beta$ the cocoercivity constant of $C$ with respect to $\|\cdot\|_P$. We proceed by showing that $(\lambda_{n})_{n\in\Nn}\subseteq [\nu_1,2-\frac{1}{2\beta}-\nu_2]$. Assume that 
  \begin{equation}\label{eqn:-16}
  0<1-\gamma_1\gamma_2\|L\|^2-\tfrac{1}{2\beta}.
  \end{equation}
  Then, there exists a positive constant $\nu_2$ such that
  \begin{equation} \label{eqn:-17}
  \nu_2<1-\gamma_1\gamma_2\|L\|^2-\tfrac{1}{2\beta}.
  \end{equation} 
  It follows from~\eqref{eqn:-23} that 
  \begin{align*} 
  	\lambda_{n} \leq  1+\frac{\gamma_{2}\|L\|^2\|\tilde{x}_{n}\|^{2}+\gamma_{1}\|L\|^2\|\tilde{y}_{n}\|^{2}}{\gamma_{1}^{-1}\|\tilde{x}_{n}\|^{2}+\gamma_{2}^{-1}\|\tilde{y}_{n}\|^{2}}
  	<
  	2-\tfrac{1}{2\beta}-\nu_2,
  \end{align*}
  where in the second inequality we used~\eqref{eqn:-17} to upper bound both $\gamma_{2}\|L\|^2$ and $\gamma_{1}\|L\|^2$.  
  We showed that~\eqref{eqn:-16} is sufficient for~\eqref{eq:-7} to hold. Consequently, \Cref{thm:Suppose-that-} completes the proof of convergence. 
  
  \eqref{lem:mu=00003D1/2:part1}: By \Cref{lem: cocoercivity lemma}, substitute $\beta$ in~\eqref{eqn:-16} with \eqref{eq:-42-43} to derive~\eqref{eq:-42-42-42-42}. 
  The algorithm follows by substituting $\alpha_{n}=1$ in  \Cref{Algorithm-2}
  when $\theta=0$ and $\mu=\frac{1}{2}$.
  
  \eqref{lem:mu=00003D1/2:part2}: Similar to \eqref{lem:mu=00003D1/2:part1}, since $l=\iota_{\{0\}}$, substitute $\beta$ in~\eqref{eqn:-16} with \eqref{eq:-30}, i.e. $\beta=\beta_{h}^{-1}\gamma_{1}^{-1}$. This choice of $\beta$ results in~\eqref{eq:-1-1-1} and the algorithm follows.
  
  \eqref{lem:mu=00003D1/2:part3}: The proof is similar to previous parts, but we show that $(\lambda_n)_{n\in\Nn}$ is uniformly bounded between $0$ and $2$, and then use \Cref{thm:Suppose-that-} with $C\equiv0$. 
\end{proof}

\end{document}